\documentclass[12pt, twoside, letterpaper]{amsart}
\usepackage{amsmath, hyperref, color, amssymb}
\usepackage{srcltx}
\usepackage[normalem]{ulem}

\usepackage{geometry, cancel}
\usepackage{enumitem}  
\newtheorem{thm}{Theorem}[section]

\newtheorem{lem}[thm]{Lemma}
\newtheorem{prop}[thm]{Proposition}

\theoremstyle{definition}
\newtheorem{defn}[thm]{Definition}

\theoremstyle{remark}
\newtheorem{rem}[thm]{Remark}
\newtheorem{exa}[thm]{Example}
\numberwithin{equation}{section}

\newcommand{\dbracc}[1]{[\kern-0.15em[ #1 ]\kern-0.15em]}
\newcommand{\dbraco}[1]{[\kern-0.15em[ #1 [\kern-0.15em[}
\newcommand{\dbraoc}[1]{]\kern-0.15em] #1 ]\kern-0.15em]}
\newcommand{\dbraoo}[1]{]\kern-0.15em] #1 [\kern-0.15em[}
\newcommand{\be}{\begin{equation}}
\newcommand{\ee}{\end{equation}}
\newcommand{\ba}{\begin{aligned}}
\newcommand{\ea}{\end{aligned}}

\begin{document}
\title[Analyticity in optimal investment and stochastic dominance]{%
On the analyticity of the value function in optimal investment and stochastically  dominant markets}
\author{Oleksii Mostovyi}
\thanks{The first author has been supported by the National Science
Foundation under grant No. DMS-1848339 and the second author has
been supported by the National Science Foundation under grant No.
DMS-1908903.}
\address{Oleksii Mostovyi, Department of Mathematics, University of
Connecticut, Storrs, CT 06269, United States}
\email{oleksii.mostovyi@uconn.edu}
\author{Mihai S\^irbu}
\address{Mihai S\^{i}rbu, Department of Mathematics, The
University of Texas at Austin, Austin, TX 78712, United States}
\email{sirbu@math.utexas.edu}
\author{Thaleia Zariphopoulou}
\address{Thaleia Zariphopoulou, Department of Mathematics and IROM, 
The University of Texas at Austin, Austin, TX 78712, United States}
\email{zariphop@math.utexas.edu}
\date{\today }
\maketitle

\begin{abstract} We study the analyticity of the value function in optimal investment with
expected utility from terminal wealth  and the relation to  stochastically dominant financial models. We identify both a class of utilities and a class of semi-martingale models for which we establish analyticity. Specifically, these utilities have completely monotonic inverse marginals, while the market models have a maximal element in the sense of infinite-order stochastic dominance. We construct two counterexamples, themselves of independent interest, which show that analyticity fails if either the utility or the market model does not belong to the respective special class. We also provide  explicit formulas for the derivatives, of all orders, of the value functions as well as their optimizers. Finally, we show that for the  set of supermartingale deflators, stochastic dominance of infinite order is equivalent to the apparently stronger dominance of second order.
\end{abstract}

\section{Introduction}

\textcolor{black}{The highest level of regularity that a function can posses is analyticity, which is also a highly desirable property that allows for classical approximation and representation results, and that emerges in many different branches of mathematics. 
 It is the ultimate level of smoothness that  is even stronger than infinite differentiability, as  the celebrated example in \cite{Osg},  given by $C^\infty$ but not analytic function $x\to e^{-1/x^2}1_{(0,\infty)}(x)$, $x\in\mathbb R$, demonstrates.}

\textcolor{black}{The regularity of value functions is a foundational and rather challenging topic in stochastic control  especially in general non-Markovian settings.  
A premier example here is the celebrated optimal investment problem in semimartingale market models. 
 Herein, we study the   {\it maximal} regularity, }  namely 
analyticity, 
of the value function
(indirect utility), $u,$ in  \textcolor{black}{such} optimal investment problems,  
given by
\begin{equation}
u(x):=\sup \limits_{X\in \mathcal{X}(x)}\mathbb{E}\left[ U(X_{T})\right]
,\quad x>0,  \label{general}
\end{equation}%
where $U:\left(0,\infty \right) \rightarrow \mathbb{R}$ is the utility
function and $\mathcal{X}(x)$ the set of admissible wealth processes
starting at $x;$ see section \ref{secMarket} for the precise definition. Problem \eqref{general} is of fundamental importance in mathematical finance and it was analyzed in a number of seminal papers (see, among
others, \cite{KLS}, \cite{KLSX}, and \cite{KS}).

 The need to study higher (even maximal) regularity for  the value function $u$  is not a purely mathematical question. The economics literature displays clearly that an investor's risk  may be encoded in  higher (than second) order behavior of the utility function $U$: see, for example,  \cite{econ4}, \cite{econ2}, \cite{econ1}, and \cite{econ3}. 
The question of whether the  indirect utility $u$ inherits all the properties of $U$ may appear a rather  abstract task at first that is asking   if the map $U\rightarrow u$ has some semigroup/invariance properties. 

 As the indirect utility, $u$, may become the \emph{direct} utility of a new optimization problem (for example, in a sequential model, or in indifference valuation, etc.), preserving the properties of $U$ to $u$  may go beyond describing the higher-order risk. Thus, \cite{Taylor1}, \cite{GR12}, \cite{GR15} use series expansions to investigate various properties in optimal investment. Therefore, the representability of an (indirect) utility as a series becomes central for the expansion-based approaches. Analyticity of the underlying utility\footnote{And not only  analyticity for coefficients of the driving dynamics as in \cite{MelamudTrub}.} plays an important role in the studies of endogenous completeness, see \cite{AndRaim} and \cite{MelamudTrub}. 
Detailed arguments why  $u$ needs to  have similar properties to  $U$ (so it can be a utility itself)  may also be found in a number of academic works; see, for example,  \cite{KLS}, \cite{KLSX}, and \cite{KS}.


\textcolor{black}{
Mathematically, it is well-known that for the value function to inherit the regularity of  original utility is an extremely subtle issue. This  is  illustrated by our counterexamples below, in the context of optimal investment. 
}
Under rather general assumptions, the authors in \cite{KS, KS2003} showed that if $%
U\in C^{1}\left( \left( 0,\infty \right) \right) $, so is $u.$ For higher-order regularity, there are two types of results. If the utility function $U$
is either power or logarithmic, the value function also inherits this form.
This is a direct consequence of homotheticity and holds under minimal model
assumptions. 

For utilities beyond homothetic ones, extra model conditions are needed even to obtain second-order differentiability of $u$ when $U$ is twice differentiable (see \cite{KS2006}). To our knowledge, no
other regularity results exist to date. We are then motivated to ask the
following question:

\textit{Can we identify both a class of semi-martingale market models and a
class of utility functions such that the value function in (\ref{general})\
retains the (highest possible) regularity of the utility function?}

We propose i) the class of market models which possess a non-zero dual
maximal element in the sense of infinite-degree stochastic dominance and ii)
the class of utilities whose inverse marginal is a completely monotonic
function (and thus analytic). We denote these classes by $\mathcal{SD}\left(
\infty \right) $ and $\mathcal{CMIM},$ respectively. 

We establish the following result: \textit{If the market model is in }$%
\mathcal{SD}\left( \infty \right) $\textit{\ and the utility is in }$%
\mathcal{CMIM},$\textit{\ then the value function }$u$ \textit{is also in} $%
\mathcal{CMIM}$ \textit{and} \textit{is, thus, analytic}$.$\textit{\ In other
words, we show that, for such market models and such analytic utilities, the
associated indirect utility inherits the analyticity and, furthermore,
remains in the same }$\mathcal{CMIM}$ class.

We also examine the necessity of these classes of models and utilities. We
provide two counterexamples, showing that the results fail outside the
family of $\mathcal{SD}\left( \infty \right) $ models and/or the utility
class $\mathcal{CMIM}.$ 

In the first counterexample, we construct a market model in $\mathcal{SD}%
\left( \infty \right) $ and an analytic, but not in $\mathcal{CMIM}$,
utility, and show that the value function is not infinitely differentiable
(and, thus, not analytic). In the second counterexample, we show that for
any non-homothetic $\mathcal{CMIM\ }$utility (actually, the utility being two-times differentiable suffices), there exists a market model 
outside the $\mathcal{SD}\left( \infty \right) $ class, for which the value
function is not even twice differentiable.

As mentioned above, under minimal model assumptions - well beyond the ones
for the $\mathcal{SD}\left( \infty \right) $ class - homothetic utilities
yield homothetic value functions. Such utilities belong to the $\mathcal{CMIM%
}$ family and are analytic, and these properties are also inherited to their
value functions. 
 Finally, we show that the class $\mathcal{SD}\left( \infty \right)$ is precisely $\mathcal{SD}\left( 2 \right) $. This result is of independent interest. It is based on a delicate simultaneous change of measure and num\'eraire combined with an approximation argument relying on a one-point compactification and the monotone class theorem.

The paper is organized as follows: in section \ref{ut-max}, we specify the settings for problem (\ref{general}). In
section \ref{secCMIM}, we discuss the background notions on complete
monotonicity and stochastic dominance and provide their characterizations. In section \ref{secMain}, we introduce the class of market models and
utilities that we propose, followed by the main
results on the analyticity of the value function together with the explicit
expressions for the primal and dual optimizers and their derivatives of all
orders, as well as other regularity results. Section \ref{exampleMihai}
provides a counterexample for non-$\mathcal{CMIM}$ utilities, while section
\ref{secExample2} contains a counterexample for non-$\mathcal{SD(\infty )}$ market models.

\bigskip

\section{The optimal investment problem}

\label{ut-max}

\subsection{The market model}

\label{secMarket} The market consists of a riskless asset, offering zero
interest rate, and $d$ traded stocks, whose price processes form a $d$%
-dimensional semi-martingale $S$ on a complete stochastic basis $(\Omega ,%
\mathcal{F},(\mathcal{F}_{t})_{t\in \lbrack 0,T]},\mathbb{P})$. Here $T\in
(0,\infty )$ is the investment horizon.

A trading strategy $H$ is a predictable and $S$-integrable process. It
generates the wealth process $X:=x+H\cdot S$, starting at $x>0$, which, for
the utilities considered herein, is taken to be non-negative. Using the
notation of \cite{KS}, we denote the set of admissible wealth processes, 
\begin{equation}\label{defX}
\begin{split}
\mathcal{X}(x):=\left \{ X:\right. & ~X_{t}=x+H\cdot S_{t}\geq 0,~~t\in
\lbrack 0,T], \\
& \left. ~for~some~S-integrable~process~H\right \} ,\quad x>0.
\end{split}
\end{equation}

Following \cite{KarKar07}, we say that a sequence $(X^{n})_{n\in \mathbb{N}%
}\subset \mathcal{X}(1)$ generates an \textit{unbounded profit with bounded
risk (UPBR)}, if the family of the random variables $(X_{T}^{n})_{n\in 
\mathbb{N}}$ is unbounded in probability, i.e., if 
\begin{equation*}
\lim \limits_{m\uparrow \infty }\sup \limits_{n\in \mathbb{N}}\mathbb{P}%
\left[ X_{T}^{n}>m\right] >0.
\end{equation*}%
If no such sequence exists, the condition of \textit{no unbounded profit
with bounded risk (NUPBR)} is satisfied. A characterization of \textit{NUPBR}
is given via the dual feasible set, $\mathcal{Y}\left( y\right) $,
introduced in \cite{KS},  
\begin{equation} \label{defY}
\begin{split}
\mathcal{Y}(y):=\left \{ Y:\right. & ~Y_{0}=y~and~XY=(X_{t}Y_{t})_{t\in
\lbrack 0,T]}~is~a~supermartingale \\
& \left. for~every~X\in \mathcal{X}(1)\right \} ,\quad y>0.
\end{split}
\end{equation}%
The elements of $\mathcal{Y}(1)$ are called super-martingale deflators, see 
\cite{KarKar07}. It was established in \cite{KarKar07} that \textit{NUPBR}
is equivalent to the existence of a strictly positive super-martingale
deflator, namely, 
\begin{equation}
\mathcal{Y}(1)~contains~a~strictly~positive~element.  \label{NUPBR}
\end{equation}

In \cite{TakaokaSchweizer} and \cite{KabKarSong16}, it was later proven that 
\textit{NUPBR }is equivalent to the existence of a strictly positive local
martingale deflator; see, also, \cite{MostovyiNUPBR}. Furthermore, it was
shown in \cite{KabKarSong16} that \textit{NUPBR} is equivalent to other
no-arbitrage conditions, such as no arbitrage of the first kind and no
asymptotic arbitrage of the\textit{\ }first kind; we refer the reader to 
\cite[Lemma A.1]{KabKarSong16} for further details.

\subsection{Utility functions}

We recall the standard class of utility functions $U:(0,\infty )\rightarrow 
\mathbb{R}$ which are strictly concave, strictly increasing, continuously
differentiable and satisfy the Inada conditions 
\begin{equation}
\lim \limits_{x\downarrow 0}U^{\prime }(x)=\infty \text{ \ }and\text{ \ }%
\lim \limits_{x\uparrow \infty }U^{\prime }(x)=0.  \label{Inada}
\end{equation}%
To facilitate the upcoming exposition, we will denote the class of all such
utility functions by $-\mathcal{C}$, in that 
\begin{equation*}
U\in -\mathcal{C}\iff -U\in \mathcal{C}.
\end{equation*}

\subsection{Primal problem and the indirect utility}

We recall the optimal investment problem from terminal wealth 
\begin{equation}
u(x):=\sup \limits_{X\in \mathcal{X}(x)}\mathbb{E}\left[ U(X_{T})\right]
,\quad x>0,  \label{primalProblem}
\end{equation}%
where $U\in -\mathcal{C}$ and $\mathcal{X}(x)$ as in (\ref{defX}).

\subsection{Dual problem and the dual function}

For any $U\in -\mathcal{C}$, its Legendre transform is given by 
\begin{equation}
V(y):=\sup \limits_{x>0}\left( U(x)-xy\right) ,\quad y>0,  \label{defV}
\end{equation}%
and, by biconjugacy, 
\begin{equation*}
-\mathcal{C}\ni U\Longleftrightarrow V\in \mathcal{C}.
\end{equation*}%
In turn, we recall the dual value function,  
\begin{equation}
v(y):=\sup \limits_{Y\in \mathcal{Y}(y)}\mathbb{E}\left[ V\left( Y_{T}\right) %
\right] ,\quad y>0.  \label{dualProblem}
\end{equation}%
with $V$ as in (\ref{defV}) and $\mathcal{Y}(y)$ as in (\ref{derY}).

\bigskip 

It was shown in \cite{KarKar07} that condition \textit{NUPBR} is necessary
for the non-degeneracy of problem \eqref{primalProblem} in that, if \textit{NUPBR}
does not hold, then, for any utility function $U$, (\ref{primalProblem}) has
either infinitely many solutions or no solution at all. Specifically, if $%
U(\infty )=\infty $, then $u(x)=\infty ,$ $x>0$. Therefore, either there is
no solution (when the supremum is not attained) or there are infinitely many
solutions (when the supremum is attained). On the other hand, if $U(\infty
)<\infty $, there is no solution.

If condition \textit{NUPBR} holds, problem (\ref{primalProblem}) has a
solution under the weak assumption that the dual value function $v$ in %
\eqref{dualProblem} is finite, i.e., $v(y)<\infty $, $y>0$. In this case,
all standard conclusions of the utility maximization theory hold; see, for
example, \cite{KostasEmery} and \cite{Mostovyi2015} for details. 

\section{Complete monotonicity and stochastic dominance}
\label{secCMIM} 

\subsection{Complete monotonicity}

Completely monotonic functions have been well-studied in the literature, see 
\cite{Widder41} and \cite{BookBernsteinFunctions} for the historic overview
of the development of the subject therein. A function $f:(0,\infty
)\rightarrow \mathbb{R}$ is called \textit{completely monotonic}, denoted by 
$f\in \mathcal{CM},$ if it has derivatives of all orders and 
\begin{equation*}
(-1)^{n}f^{(n)}(x)\geq 0,\text{ }x>0\text{ \ and \ }n=0,1,2,\dots 
\end{equation*}%
Whenever needed, we extend $f$ to $[0,\infty )$ setting $f(0):=\lim%
\limits_{x\downarrow 0}f(x)$, where $f(0)\leq \infty$.

The celebrated Bernstein theorem (see \cite[Theorem 1.4]%
{BookBernsteinFunctions} or \cite[Theorem 12b]{Widder41}) gives a
characterization of completely monotonic functions, stating that $f\in 
\mathcal{CM}$ if and only if 
\begin{equation}\label{254}
f(x)=\int\limits_{0}^{\infty }e^{-xz}d\mu (z),\quad x\geq 0,
\end{equation}%
where $\mu $ is a nonnegative sigma-finite measure on $[0,\infty )$ such
that the integral converges for every $x>0.$

\begin{defn}\label{defD}We define $\mathcal{D}$ to be the class of
functions $W:\left[ 0,\infty \right) \rightarrow \mathbb{R}$, which satisfy \begin{enumerate}
\item $%
-W^{\prime }\in \mathcal{CM}$,
\item
 $W^{\prime }\left( \infty \right) =0$.
 \end{enumerate}
\end{defn}
The reader should note that the definition above is related, but not the same, to what is called in literature a Bernstein functions, see, e.g., \cite[p. 15]{BookBernsteinFunctions}. Bernstein functions  would assume bounds on $W$, but no Inada-type conditions on $W'$.
For a   $W\in  \mathcal{D}$,  we have  $W'(y) = -\int\limits_0^\infty e^{-yz}d\mu(z)$  from the Bernstein representation characterization of completely monotonic functions.  We then deduce that 
$$W'(0+):=\lim\limits_{y\downarrow 0}W'(y) = -\mu([0,\infty))\quad and  \quad W'(\infty):=\lim\limits_{y\uparrow \infty}W'(y) = -\mu(\{0\}).$$
Therefore, the definition of $\mathcal D$  dictates that the measure $\mu$ has no mass at $z=0$, to satisfy  \begin{equation}\label{151}\mu (\{0\})=-W'(\infty)=0.\end{equation} 
We note that the Inada-type  condition $W'(0)=-\infty$ holds if and only if $\mu([0,\infty)) = \mu ((0,\infty))= \infty$,  not assumed for  $W\in \mathcal D$.

\subsection{Monotonicity of finite order}
A weaker notion of complete monotonicity is the monotonicity of finite order. We adopt the slightly more restrictive definition of monotonicity of order $n$ in the paper \cite{83-lorch-neumann} and not the somewhat weaker definitions in the earlier works \cite{56-williamson} and \cite{71-muldoon}. 

A function $f:(0,\infty )\rightarrow \mathbb{R}$ is called 
monotonic of (finite) order $n,$ denoted by $f\in \mathcal{CM}(n),$ if it
has derivatives of order $k=1,2,...,$ $n$ and 
\begin{equation*}
(-1)^{k}f^{(k)}(x)\geq 0,~\text{\ }x>0\quad and\quad k=0,1,2,\dots ,n.
\end{equation*}%
As in the $\mathcal{CM}$ case, whenever needed, we extend $f$ to $[0,\infty )$ by $f(0):=
\lim \limits_{x\downarrow 0}f(x)$, where $f(0)\leq \infty$.
In analogy to the class $\mathcal{D},$ we introduce the following definition.

\begin{defn}\label{defDn}
For $n\geq 1$, 
we define $\mathcal{D(}n\mathcal{)}$ to be the class
of functions $W:\left[ 0,\infty \right) \rightarrow \mathbb{R}$, which satisfy
\begin{enumerate}
\item $-W^{\prime }\in \mathcal{CM}(n-1)$,\item $W^{\prime }\left( \infty
\right) =0.$ 
\end{enumerate}
\end{defn}

We note that $W\in \mathcal D(n)$ is not necessarily strictly decreasing by definition. While we assume that $W'(\infty)=0$ (both to make it similar to the $n=\infty$ case and to simplify the upcoming definition of stochastic dominance),  we do not impose the  condition $W'(0)=\infty$ for $W\in \mathcal D(n)$, nor assume that such a $W$ is bounded below.

\begin{prop}\label{3.3} Fix $W\in \mathcal D(n)$, $n\in\{2,3,\dots\}$. Then,  
\begin{equation}   \label{high-order-null} W'(\infty)=W''(\infty)= \dots =W^{(n-1)}(\infty)=0
\end{equation} and 
$$0\leq -W'(y_1)= \int\limits_{y_1}^{\infty} \dots 
\int\limits_{y_{n-1}}^{\infty} (-1)^nW^{(n)}(y_n)    dy_n 
\dots  dy_2 <\infty,\quad y_1>0.$$
Therefore, any $W\in \mathcal{D}(n)$ has the representation
\begin{equation}
\label{Vn}W(y)=W(y_0)+  \int\limits_y ^{y_0} \int\limits_{y_1}^{\infty} \dots 
\int\limits_{y_{n-1}}^{\infty} (-1)^nW^{(n)}(y_n)    dy_n 
\dots  dy_2 dy_1, \ \ y>0.
\end{equation}
For each fixed $y_0>0$, the above representation holds.\end{prop}
\begin{proof} As $W\in \mathcal D(n)$, we have $(-1)^kW^{(k)}(y)\geq 0$ for $k=1,\dots, n$ and $W'(\infty)=0$. Assume now that
$W^{(k)}(\infty)=0$ for some $k\leq n-2$. Since $ (-1)^{k+2}W^{(k+2)} (y)\geq 0,$ we conclude that the  function
 $ y\rightarrow (-1)^{k+1}W^{(k+1)} (y)\geq 0$  is decreasing.  Next, assume that $W^{(k+1)}(\infty)\not= 0$,  so 
  $$(-1)^{k+1}W^{(k+1)} (\infty) >0.$$
  This, however, would contradict the monotonicity of  
  $y\rightarrow  (-1)^kW^{(k)}(y)$, which is decreasing, and the  assumption that $W^{(k)}(\infty)=0$.  An inductive argument completes the proof.  \end{proof}
\subsection{Stochastic dominance of finite order}

Let $F$ and $G$ be two cumulative distribution functions with supports on $%
\mathbb{R}_{+}=[0,\infty )$. We recall that {\it $F$ stochastically dominates $G$ 
in the first order} if 
$$F(y)\leq G(y),\quad y\geq 0.$$ 
%

To define stochastic dominance of higher orders, following, for example, \cite{Thistle},  we set 
\begin{equation}
F_{1}=F\text{ \  \  \ and \  \ }F_{i}(y)=\int\limits_{0}^{y}F_{i-1}(z)dz,\quad
i=2,3\dots .  \label{F-n}
\end{equation}%
Since $0\leq F\leq 1$, the integrals are well defined. The functions $G_{i}$
are defined similarly. 
Next, we depart slightly from the definition customary in the literature, e.g.,  in \cite{Thistle}, see also \cite{01-zhang} and \cite{whitmore3rd}.  On the one hand, we use a somewhat weaker definition, while, on the other, we can treat unbounded supports. More comments follow the definition. 
\begin{defn}\label{def:dominance-fin}
For any $n\geq 1$, we say that $F$ stochastically dominates $G$ in the sense of the  $n$-th order, and denote $F {\succeq}_n G$, if  $F_{n}(y)\leq G_{n}(y)$,  $y\geq 0$.  For two random variables $\xi, \eta \geq 0$ we say that  $ \xi {\succeq}_{n} \eta$ if  $F_{\xi}  {\succeq}_{n} F_{\eta}$. \end{defn}
\begin{rem} For $n\geq 3$, it is customary in the literature,  in order to define 
$F {\succeq}_n G$, to both 
\begin{enumerate}\item  assume that $F$ and $G$ are supported on a finite interval $[0,b]$, 
\item  have the additional condition 
$F_k(b)\leq G_k(b)$, $k = 1,\dots,n-1$. 
\end{enumerate}
Our definition by-passes both points above  since   we will only use  a restrictive set of  ``test'' functions, namely $ \mathcal D(n)$. For such test functions, condition \eqref{high-order-null}  ensures that we do not (even formally) need the extra assumption.  Furthermore, our definition works well for $n\geq 3$ for measures fully  supported on the  $[0,\infty)$ that we need.
\end{rem}

\begin{prop}\label{char:fin-order} Consider two non-negative random variables $\xi$ and $\eta$. Fix $n\geq 2$.  Then, the following conditions are equivalent:
\begin{enumerate}
\item $ \xi {\succeq}_{n} \eta$,
\item $\mathbb{E}[W(\xi)]\leq \mathbb E [W(\eta )]$
for every   function  $W \in \mathcal D(n)$, such that  $W(\infty)>-\infty,$ (i.e., $W$ is bounded from below),
\item $\mathbb{E}[W(\xi)]\leq \mathbb E [W(\eta )]$
for every   function  $W \in \mathcal D(n)$
such that
$\mathbb E [W^-(\xi )]<\infty$ and $ \mathbb E [W^-(\eta )]<\infty.$ \end{enumerate}
\end{prop}
\begin{proof}
 If $W\in \mathcal D (n)$ is bounded below, we  will suppose that $W(\infty)=0$, without loss of generality. For $y_0=\infty$,  representation \eqref{Vn} becomes 
\begin{equation}
\begin{split}
\label{Vnbd}W(y)&=  \int_y ^{\infty} \int\limits_{y_1}^{\infty} \dots ~
\int\limits_{y_{n-1}}^{\infty} (-1)^nW^{(n)}(y_n)    dy_n 
~\dots  dy_2 dy_1\\
&=\int\limits_{\mathbb R^n_{+}} 1_{\{ y\leq y_1\leq \dots\leq y_n\}}(-1)^nW^{(n)}(y_n)    dy_n 
~\dots  dy_2 dy_1,\quad y>0.
\end{split}
\end{equation}
Therefore, Fubini's theorem yields
\begin{equation}\label{11271}
\begin{split}
\mathbb E[W(\xi)]& = \mathbb E\left[ \int\limits_{\mathbb R^n_{+}} 1_{\{ \xi\leq y_1\leq \dots\leq y_n\}}(-1)^nW^{(n)}(y_n)    dy_n 
~\dots  dy_2 dy_1 \right] \\
&= \int\limits_{0}^\infty  \left (\int\limits_{\mathbb R^{n-1}_{+}} \mathbb E\left[ 1_{\{ \xi\leq y_1\leq \dots\leq y_n\}}\right]dy_1 
~\dots  dy_{n-1} \right ) (-1)^nW^{(n)}(y_n)dy_n.
\end{split}
\end{equation}
Fix  $y_n$.  Using  the cdf  $F$ of  $\xi$ we can rewrite
\begin{displaymath}
\begin{split}\int\limits_{\mathbb R^{n-1}_{+}} \mathbb E\left[ 1_{\{ \xi\leq y_1\leq \dots\leq y_n\}}\right]dy_1 
~\dots  dy_{n-1} & =  \int\limits_{0}^{y_n}\dots\int\limits_{0}^{y_2}\mathbb P\left[  \xi\leq y_1\right]dy_1 
~\dots  dy_{n-1} \\
&= \int\limits_{0}^{y_n}\dots\int\limits_{0}^{y_2}F(y_1)dy_1 
~\dots  dy_{n-1} = F_n(y_n),
\end{split}
\end{displaymath}
where we have used that $\xi\geq 0$.
Together with \eqref{11271}, we obtain
\begin{equation}
\label{Fubini-n}\mathbb E[W(\xi)]=\int\limits_0^{\infty} W(y)dF(y)= \int\limits_0^{\infty} (-1)^n W^{(n)} (y_n) F_n(y_n) dy_n.
\end{equation}
This shows that (1) and   (2) above are equivalent. 

To show $(2)\Rightarrow (3)$, for a general $W\in \mathcal D(n)$ in (3),  first assume that $W(1)=0$, without any loss of generality.  
Next, one has to use a smooth cut-off of the $n$-th derivative of $W$ away from infinity.  More precisely, consider an increasing sequence of functions
$$0\leq f_i(y)\nearrow 1, 0<y<\infty$$ and such that $$supp (f_i)\subset (\frac 1i, i).$$  Then we set  
$$W^{(n)}_i (y) : = f_i(y)W^{(n)}(y), y>0,$$
and $$W_i(1)=0, W'(\infty)=...=W^{(n-1)}(\infty)=0.$$
We therefore  recover  aol the lower order  derivatives, up to the first order, using the computation in Proposition \ref{3.3}:
$$0\leq -W'_i(y_1)= \int\limits_{y_1}^{\infty} \dots 
\int\limits_{y_{n-1}}^{\infty} (-1)^n f_i(y_n)W^{(n)}(y_n)    dy_n 
\dots  dy_2 <\infty,\quad y_1>0,$$
so 
$$0\leq  -W'_i(y_1)\nearrow -W'(y_1), 0<y_1<\infty.$$
We integrate the above relation to conclude that
$$W_i^+ (y)=1_{\{y\leq 1\}}W_i (y)=1_{\{y\leq 1\}}\int _y^1 -W'_i(y_1)dy_1\nearrow  1_{\{y\leq 1\}}\int _y^1 -W'(y_1)dy_1=W^+(y),$$
and
$$W_i^- (y)=1_{\{y>1\}}(-W_i (y))=1_{\{y>1\}}\int _1^y -W'_i(y_1)dy_1\nearrow  1_{\{y> 1\}}\int _1^y -W'(y_1)dy_1=W_+(y).$$
Using (2) for the bounded test functions $W_i\in \mathcal D(n)$, we can pass to the limit separately for the positive and negative parts to conclude that $(2)\Rightarrow (3)$.
\end{proof}

\subsection{Stochastic dominance of infinite degree}
The infinite-order stochastic dominance is, intuitively, defined by letting $n\uparrow \infty$ in Definition \ref{def:dominance-fin}. This, however, has to be done carefully. We again depart from \cite{Thistle} for our definition. 

To provide some intuition, we first note that, for each $z>0$, 
the exponential function  $W(y )=e^{-zy}$, $y>0$, is in $\mathcal D (n)$, for every $z>0$ and $n\geq 1$. For every $z>0$, relation  \eqref{Fubini-n} reads
$$\mathbb E [e^{-z\xi}]=\int\limits_0^{\infty} e^{-zy}dF(y) =\int\limits_0^{\infty} z^{n} e^{-zy}F_n(y)dy,\quad n\geq 1. $$ 
Therefore, if for any $n$, no matter  how large, we have $ F{\succeq}_n G$, then the exponential moments of the two distributions compare, for all  positive values of $z$.  It thus appears to us that the weakest possible form of dominance, obtaining by letting $n\uparrow \infty$, is the one below.
\begin{defn}\label{def:dominance-inf} Consider two  cumulative distributions    $F$ and $G$ on $[0, \infty)$. We say that $F$ dominates $G$ in {\it infinite degree stochastic dominance}, and denote by $F  {\succeq}_{\infty}G$, if 
$$\int\limits_0^{\infty} e^{-zy}dF(y)\leq  \int\limits_0^{\infty} e^{-zy}dG(y), \quad  z>0.$$
For nonnegative random variables $\xi$ and $\eta$, we say that $\xi $ dominates $%
\eta $ in \textit{infinite-order stochastic dominance}, and denote $ \xi {\succeq}_{\infty} \eta$  if  $F_{\xi}  {\succeq}_{\infty} F_{\eta}$, that is
$$\mathbb E\left [e^{-z \xi}\right]\leq \mathbb E\left [e^{-z\eta}\right],\quad  z>0.$$\end{defn}
\begin{rem}
To the best of our knowledge, 
the name  of {\it infinite-order stochastic dominance} first    appeared  in \cite{Thistle} (but for a somewhat less precise definition), whereas  \cite{BrockettGolden}  does not  use the  specific name of infinite-order dominance. 
\end{rem}
 Below, we provide a characterization of infinite-order stochastic
dominance. 
\begin{prop}\label{char:inf-order} Consider two non-negative random variables $\xi$ and $\eta$. Then, the following conditions are equivalent:
\begin{enumerate}
\item $ \xi {\succeq}_{\infty} \eta$,
\item $\mathbb{E}[W(\xi)]\leq \mathbb E [W(\eta )]$
for every function  $W \in \mathcal D $, such that $W(\infty) > -\infty$, i.e., $W$ is bounded from below,
\item $\mathbb{E}[W(\xi)]\leq \mathbb E [W(\eta )]$
for every function  $W \in \mathcal D $  such that
\begin{equation}\label{252}
\mathbb E [W^-(\xi )]<\infty\quad  and \quad \mathbb E [W^-(\eta )]<\infty.
\end{equation}
\end{enumerate}
\end{prop}
%

\begin{proof} 
$(2) \Rightarrow (1)$. Let us suppose that $(2)$ holds and assume by contradiction that there exists $z>0$, such that 
\begin{equation}\label{2561}
\mathbb E\left[ e^{-z \xi}\right] > \mathbb E\left[e^{-z \eta} \right].
\end{equation}
Let us consider a measure $\mu$  such that 
\begin{equation*}\mu(\{0\}) = 0,\quad  \quad \mu((0,\infty)) = \infty,\quad and \quad \int\limits_0^\infty\int\limits_0^\infty e^{-yt}\mu(dt) {dy=\int\limits_0^\infty \frac 1t\mu(dt)}<\infty.
\end{equation*}
For example \begin{displaymath}
\mu(dt) = \left\{\begin{array}{ll}t^{-\tfrac 12}dt,& t\geq 1\\
0  {\ dt},&t\in(0,1)
\end{array}
\right.,
\end{displaymath}
works. Next, 
let us  define $W$  as 
\begin{equation*}
W(y) := \int\limits_y^\infty\int\limits_0^\infty e^{-zt}\mu(dt)dz,\quad y>0.
\end{equation*}
Then $W\in\mathcal D$ and is bounded from below by $0$ and from above by $\int\limits_0^\infty\int\limits_0^\infty e^{-yt}\mu(dt) {dy}<\infty$. 
Now, for a constant $c>0$, we set   $$\nu_c := \mu + c\delta_z,$$
 where $\delta_z$ is a delta function centered at $z$. 
We further define  
 \begin{equation*}
W_c(y) := \int\limits_y^\infty\int\limits_0^\infty e^{-zt}\nu_c(dt)dz,\quad y>0.
\end{equation*}
One can see that $W_c\in\mathcal D$ and is bounded from below by $0$. 
By Tonelli's theorem, we have 
\begin{displaymath}
W_c(y) = \int_0^\infty e^{-yt}\frac{ {\nu_c}(dt)}{t}, \quad y >0,
\end{displaymath}
and thus
\begin{displaymath}\begin{split}
\mathbb E\left[ W_c(\xi)\right] &= \mathbb E\left[ \int_0^\infty e^{-\xi t}\frac{\nu_c(dt)}{t} \right] = \int_0^\infty \mathbb E\left[ e^{-\xi t} \right] \frac{\nu_c(dt)}{t}
\\
&= \int_0^\infty \mathbb E\left[ e^{-\xi t} \right] \frac{\mu(dt)}{t} + \frac{c}{z} \mathbb E\left[e^{-\xi z}\right] =\mathbb E\left[ W(\xi)\right] +  \frac{c}{z} \mathbb E\left[e^{-\xi z}\right].
\end{split}
\end{displaymath}
Therefore, from \eqref{2561}, and for a sufficiently large $c$, we obtain that 
\begin{displaymath}
\mathbb E\left[ W_c(\xi)\right] = \mathbb E\left[W(\xi)\right] +  \frac{c}{z} \mathbb E\left[e^{-\xi z}\right] > 
\mathbb E\left[ W(\eta)\right] +  \frac{c}{z} \mathbb E\left[e^{-\eta z}\right] = \mathbb E\left[ W_c(\eta)\right],
\end{displaymath}
which contradicts  $(2)$. 

$(3) \Rightarrow (2)$ is trivial. Therefore, it remains to show
$(1) \Rightarrow (3)$. Let us consider $\xi$ and $\eta$ satisfying $(1)$ and $W\in\mathcal D$  such that \eqref{252} holds.  Then $-W'$ admits the representation \eqref{254} for some nonnegative sigma-finite measure $\mu$ on $[0,\infty)$, satisfying \eqref{151}, and such that the integral in \eqref{254} converges for every $x>0$. 
Let us also fix $y_0>0$. Using the Bernstein representation and Tonelli's theorem, we get
\begin{equation}
\label{253}
\begin{split}
&W(y) - W(y_0) = \int\limits_y^{y_0}\left(-W'(x)\right)dx = \int\limits_y^{y_0}\int\limits_{0}^\infty e^{-xt}\mu(dt)dx \\
&=\int\limits_{0}^\infty \int\limits_y^{y_0}e^{-xt}dx\mu(dt) = \int\limits_{0}^\infty \left( e^{-yt} - e^{-y_0t}\right)\frac{\mu(dt)}{t}.
\end{split} 
\end{equation}
Therefore, using \eqref{252} and Tonelli's theorem, we obtain 
\begin{displaymath}\begin{split}
-\infty& <\mathbb E\left[\left(W(\xi) - W(y_0)\right)1_{\{\xi\geq y_0\}}\right] = 
\mathbb E\left[\left(\int\limits_{0}^\infty \left( e^{-\xi t} - e^{-y_0t}\right)\frac{\mu(dt)}{t}\right)1_{\{\xi\geq y_0\}}\right]
\\
= &\mathbb E\left[\int\limits_{0}^\infty \left( e^{-\xi t} - e^{-y_0t}\right)1_{\{\xi\geq y_0\}}\frac{\mu(dt)}{t}\right]  =
\int\limits_{0}^\infty\mathbb E\left[ \left( e^{-\xi t} - e^{-y_0t}\right)1_{\{\xi\geq y_0\}}\right]\frac{\mu(dt)}{t}
\leq 0.
\end{split}
\end{displaymath}
Consequently, using \eqref{151} and the monotone convergence theorem, we conclude that 
\begin{equation}\label{255}\begin{split}
- \infty&<\mathbb E\left[\left(W(\xi) - W(y_0)\right)1_{\{\xi\geq y_0\}}\right]=\int\limits_{0}^\infty\mathbb E\left[ \left( e^{-\xi t} - e^{-y_0t}\right)1_{\{\xi\geq y_0\}}\right]\frac{\mu(dt)}{t}\\&= \lim\limits_{ n\to\infty}\int\limits_{1/n}^n \mathbb E\left[ \left(e^{-\xi t} - e^{-y_0t} \right)1_{\{\xi \geq y_0\}}\right]\frac{\mu(dt)}{t}\leq 0.
\end{split}
\end{equation} 
Similarly, we obtain 
\begin{equation}
\label{256}\begin{split}
0\leq &\mathbb E\left[\left(W(\xi) - W(y_0)\right)1_{\{\xi< y_0\}}\right]=
\mathbb E\left[ \int \limits_{0}^\infty \left( e^{-\xi t} - e^{-y_0t}\right)1_{\{ \xi < y_0\}}\frac{\mu(dt)}{t}\right] \\
&=
\mathbb E\left[\lim\limits_{n\to \infty} \int \limits_{1/n}^n \left( e^{-\xi t} - e^{-y_0t}\right)1_{\{ \xi < y_0\}}\frac{\mu(dt)}{t}\right] = \lim\limits_{n\to \infty} \mathbb E\left[\int \limits_{1/n}^n \left( e^{-\xi t} - e^{-y_0t}\right)1_{\{ \xi < y_0\}}\frac{\mu(dt)}{t}\right] \\
&= \lim\limits_{n\to \infty}  \int \limits_{1/n}^n\mathbb E\left[\left( e^{-\xi t} - e^{-y_0t}\right)1_{\{ \xi < y_0\}}\right] \frac{\mu(dt)}{t},
\end{split}
\end{equation}
where the latter limit might be finite or not. Combining \eqref{255} and \eqref{256}, we assert that 
\begin{equation*}
-\infty<\mathbb E\left[ W(\xi)  \right]
=W(y_0) + \lim\limits_{n\to \infty}  \int \limits_{1/n}^n\left( \mathbb E\left[e^{-\xi t} \right]- e^{-y_0t}\right) \frac{\mu(dt)}{t},
\end{equation*}
and a similar representation holds for $\eta$. Therefore, if $ \xi {\succeq}_{\infty} \eta$ and \eqref{252} holds, we have 
\begin{displaymath}\begin{split}
-\infty< \mathbb E\left[ W(\xi)  \right]
&=W(y_0) + \lim\limits_{n\to \infty}  \int \limits_{1/n}^n\left( \mathbb E\left[e^{-\xi t} \right]- e^{-y_0t}\right) \frac{\mu(dt)}{t} \\
&\leq 
W(y_0) + \lim\limits_{n\to \infty}  \int \limits_{1/n}^n\left( \mathbb E\left[e^{-\eta t} \right]- e^{-y_0t}\right) \frac{\mu(dt)}{t} =  \mathbb E\left[ W(\eta)  \right],
\end{split}
\end{displaymath}
which implies $(3)$. 
This completes the proof of the proposition. 
\end{proof}
\section{Main results}\label{secMain} 


\subsection{The $\mathcal{CMIM}$ and $\mathcal{CMIM}\left(
n\right) $ utilities}

As utility-based preferences are invariant under positive linear transformations of the form $U^*(x) = aU(x) + b$, $a>0$, and in view of the importance of the marginal utility in many problems, it is natural to define a utility function through its derivative. Additionally, it has been observed that the most widely used utility functions have completely monotonic marginals, see \cite{BrockettGolden}. In the present paper, we investigate a class of functions, whose {\it inverse marginals} are completely monotonic. This is particularly natural in view of the overall importance of the duality approach to the expected utility maximization.

We start with the following definition.

\begin{defn}\label{defCMIM} We define  the $\mathcal {CMIM}$ to be the class of utility functions $U\in - \mathcal C$ for which their inverse marginal $(U')^{-1}\in \mathcal {CM}$. 
\end{defn}
From Bernstein's theorem, we deduce that if $U \in\mathcal{CMIM}$, then we have the representation
\begin{equation}\label{bernstein}
(U')^{-1}(y) = \int\limits_0^\infty e^{-yz}\mu(dz),\quad y>0,
\end{equation}
where $\mu$ is a nonnegative measure, such that the integral converges for every $y>0$.

We stress that the Inada conditions (\ref{Inada}) dictate that the
underlying measure $\mu $ must satisfy $\mu \left( \left \{ 0\right \} \right)
=0$ and $\mu \left( (0,\infty \right) )=\infty .$ Indeed, $U^{\prime }\left(
0\right) =\infty $ yields $\mu \left( \left \{ 0\right \} \right) =0$ while $%
U^{\prime }\left( \infty \right) =0$ yields $\mu ((0,\infty ))=\infty .$

\begin{exa} Here we show that standard utilities are included.
\begin{enumerate}\item
$U(x)=\log x$, $x>0$. Then, $\left( U^{\prime }\right) ^{-1}(y)=\frac{1%
}{y}\in \mathcal{CM}$ and we have  $$\left( U^{\prime }\right)
^{-1}(y)=\int\limits_{0}^{\infty }e^{-yz}dz,\quad y>0.$$ 

\item $U(x)=\frac{x^{p}}{p}$, $x>0$, $p<1$, $p\neq 0$. Then, $\left( U^{\prime }\right)
^{-1}(y)=y^{-\frac{1}{1-p}}$ $\in \mathcal{CM}$ and  with $q = -\frac{p}{1-p}$ (i.e., such that $\frac{1}{p}+\frac{1}{q}=1$),  we have $$\left( U^{\prime
}\right) ^{-1}(y)=\frac{1}{\Gamma (1-q)}\int\limits_{0}^{\infty }e^{-yz}z^{-q}dz,\quad y>0,$$ 
where $\Gamma $ is the
Gamma function, see \cite[p. viii, formula (2)]%
{BookBernsteinFunctions}.
\end{enumerate}
\end{exa}
Assuming less regularity on the utility function but keeping monotonic structure up to finite order leads to the following definition.
\begin{defn}\label{defCMIMn}
For $n\in\{2,3,\dots\}$, 
we say that a utility function $U\in - \mathcal{C}$ is in the  $\mathcal{CMIM}(n)$ class  if its inverse marginal
is completely monotonic of order $n-1$, that is
 $\left( U^{\prime }\right) ^{-1}\in \mathcal{CM}(n-1).$
\end{defn} 


Recalling Definitions \ref{defD} and \ref{defCMIM}, and denoting by $V$ the convex conjugate of $U$ in the sense of  \eqref{defV}, we deduce that 
\begin{equation*}
U\in \mathcal{CMIM}\iff V\in \mathcal{D}\cap \mathcal{C}.
\end{equation*}%
Likewise, from Definitions \ref{defDn} and \ref{defCMIMn}, we get 
\begin{equation*}
U\in \mathcal{CMIM}\left( n\right) \iff V\in \mathcal{D}\left( n\right) \cap 
\mathcal{C}.
\end{equation*}

\subsection{The class $\mathcal{SD}\left( \infty \right) $ and $\mathcal{SD}%
\left( n\right) $ of market models}

The stochastic dominance had to be formally defined separately for finite and infinite $n$ (see Definitions \ref{def:dominance-fin} and \ref{def:dominance-inf}). However, for the associated market models, now, we can give a unified definition (for both finite and infinite degree) below. 
\begin{defn} 
 Fix $n\in \{2, 3,  \dots\}\cup \{\infty\}$. We say that the financial model satisfies condition $\mathcal{SD}(n)$ if there exists 
 $\widehat Y \in \mathcal Y(1)$ such that
 $\widehat Y _T{\succeq}_n Y_T$ for every $Y\in \mathcal Y(1)$.
 \end{defn}
%
%

In what follows, we will use the terminologies ``market model in $\mathcal{SD}%
(n)$ class'' and ``$\mathcal{SD}(n)$-model'' interchangeably. 
In view of Propositions \ref{char:fin-order} and \ref{char:inf-order}, we
have the following result, for \textit{both} infinite and finite orders.

 \begin{lem} \label{lem:same-dual}
Fix $n\in \left \{ 2,3,...\right \} \cup \{ \infty \}.$  Assume that the model satisfies condition  $\mathcal{SD}(n)$  and that  $U\in \mathcal{CMIM}(n)$, thus $V\in \mathcal D (n)\cap \mathcal C$. Then, the  dual value function has the representation
 \begin{equation}\label{v-dual}
 v(y)=\mathbb E[V(y\widehat Y_T)], \quad y>0.
 \end{equation}
\end{lem}

This result yields the key property that, up to a multiplicative constant, the dual problem admits the 
\textit{same optimizer,} $\widehat{Y}$, for \textit{any} initial $y>0$. 
Representation (\ref{v-dual}) can be thought of as a \textit{relaxation} of
the notion of model \textit{completeness}, in the following sense: while a
market model in $\mathcal{SD}(n)$ is, in general, incomplete from the point
of view of replication, it does behave like a complete one from the point of
view of optimal investment, if the utility function $U\in $ $\mathcal{CMIM}%
(n)$.

\subsection{Main theorems}

We will assume that 
\begin{equation}\label{finDual}
v(y)<\infty, \quad y>0.
\end{equation}
We recall that the above condition is the canonical integrability one on the
dual value function that is necessary and sufficient for the standard assertions of the
utility maximization theory to hold; see \cite{KS2003} (see, also, \cite%
{KostasEmery} for the formulation without {\it NFLVR} and \cite{Mostovyi2015} for
the formulation with intermediate consumption and stochastic utility, where 
\eqref{finDual} is combined with the finiteness of the primal value
function). We also recall that for $U\in-\mathcal C$, under \eqref{NUPBR} and \eqref{finDual}, for every $x,y>0$,
there exist unique optimizers, $\widehat{X}(x)\in \mathcal{X}(x)$ and $%
\widehat{Y}(y)\in \mathcal{Y}(y),$ for the primal \eqref{primalProblem} and dual
\eqref{dualProblem} problems, respectively. This is a consequence of the abstract theorems in \cite{Mostovyi2015}.

\begin{thm}\label{thm1} Consider a financial model for which  \eqref{NUPBR} holds, and which is in $\mathcal{SD}(\infty)$. 
For $U\in \mathcal{CMIM}$, consider the optimal investment problem  \eqref{primalProblem} and assume that \eqref{finDual} holds. Then, the following assertions hold:

\begin{enumerate}\item
The value function $u\in \mathcal{CMIM}$ and is thus analytic. 
\item The dual value function $v\in \mathcal{D}\cap \mathcal C$. Furthermore, for $n \in\{1,2,\dots\}$, we have 
\begin{equation}\label{derv}
\begin{split}
(-1)^nv^{(n)}(y)  =&~(-1)^n\mathbb E\left[ V^{(n)}(\widehat Y_T(y))\left(\frac{\widehat Y_T(y)}{y}\right)^n\right] \\
=&~(-1)^n\mathbb E\left[ V^{(n)}(y\widehat Y_T)\left(\widehat Y_T\right)^n\right]\\
=&~\mathbb E\left[ \int\limits_0^\infty e^{-zy\widehat  Y_T }z^{n-1}\widehat Y_T^n\mu(dz)\right]\in(0,\infty),\quad y>0.
\end{split}
\end{equation} 

\item 
The $\mathcal{CM}$ function $-v'$  admits the Bernstein representation
\begin{equation}\label{bernRepv}
v'(y) = -\int\limits_0^\infty e^{-yz}\nu(dz),\quad y>0,
\end{equation}
for some sigma-finite measure $\nu$ supported in $(
0,\infty ) $, such that $$\nu\left( \left \{ 0\right \}\right) =0\quad {and}\quad \nu\left( \left( 0,\infty \right) \right) =\infty ,$$ and which satisfies
\begin{equation}\label{propertyNu}
\lim \limits_{n\uparrow \infty }\int\limits_{(0,z]}(-1)^{n+1}v^{(n+1)}\left(\frac{n}{\rho 
}\right)\left(\frac{n}{\rho }\right)^{n+1}d\rho =\frac{\nu\left( \left( 0,z\right]
\right) +\nu\left( \left( 0,z\right) \right) }{2},\quad z>0.
\end{equation}
\item 
For $n\geq 2$ and $f(y) := -\frac{1}{v''(y)}$, $y>0$, 
we  have
\begin{equation}\label{deru}
\begin{split}
u^{(n)}(x) =\sum \frac{(n-2)!}{k_1!1!^{k_1}\dots k_{n-2}!{(n-2)!}^{k_{n-2}}}f^{(k_1 +  \dots +k_{n-2})}(u'(x))\prod\limits_{j=1}^{n-2}\left(
{u^{(j+1)}(x)}
\right)^{k_j},\\
x>0,
\end{split}
\end{equation}
where  the sum is over all $n$-tuples of nonnegative integers $(k_1, \dots, k_{n-2})$ satisfying 
  $\sum\limits_{i=1}^{n-2}ik_{i}=n-2$. 
  \end{enumerate}
\end{thm}

The following theorem specifies the derivatives of optimizers of all orders for both primal and dual problems.
\begin{thm}\label{thm2}
Under  the conditions of Theorem \ref{thm1}, the following assertions hold:
\begin{enumerate}\item
The primal and dual optimizers are related via
\begin{equation}\label{7113}
\widehat X_T(x) = \int\limits_0^\infty e^{ -zu'(x)\widehat Y_T}\nu(dz),\quad x>0.
\end{equation}
\item
Their derivatives are given by
\begin{equation}\label{derX}
\begin{split}
\widehat X^{(1)}_T(x):=\lim\limits_{h\to 0}\frac{\widehat X_T(x+h)-\widehat X_T(x)}{h} 
&= -V''(u'(x)\widehat Y_T)u''(x)\widehat Y_T\\
 &=~\frac{u''(x)}{U''(\widehat X_T(x))}\widehat Y_T>0,
\end{split}
\end{equation}
\begin{equation}\label{derY}
\widehat Y_T^{(1)}(y):=\lim\limits_{h\to 0}\frac{\widehat Y_T(y+h)-\widehat Y_T(y)}{h} = \frac{\widehat Y_T(y)}{y} = \widehat Y_T.
\end{equation}
\item
Recursively, for every $n\geq 2$, the higher-order derivatives are given by 
\begin{displaymath}\begin{split}
\widehat X_T^{(n)}(x):=&\lim\limits_{h\to 0 }\frac{\widehat X_T^{(n-1)}(x + h)-\widehat X_T^{(n-1)}(x)}{h} \\
=&\sum \frac{n!}{k_1!1!^{k_1}k_2!2!^{k_n}\dots k_n!n!^{k_n}}\left(-V^{(1+k_1 +  \dots +k_n)}\left(u'(x)\widehat Y_T\right)\right)\prod\limits_{j=1}^n\left(
{u^{(j+1)}(x)\widehat Y_T}
\right)^{k_j},\\
\end{split}
\end{displaymath}
where  the sum is over all $n$-tuples of nonnegative integers $( k_1, \dots, k_n)$ satisfying 
$\sum\limits_{i = 1}^n ik_i = n$. Trivially,  
$$\widehat Y_T^{(n)}(y):= \lim\limits_{h\to 0 }\frac{\widehat Y_T^{(n-1)}(y + h)-\widehat Y_T^{(n-1)}(y)}{h} = 0.$$
\end{enumerate}
The limits  for the derivatives of  the optimizers above are in probability. However,  because of the multiplicative structure of the dual $\widehat Y(y)=y\widehat Y$,  the limits above can be understood in the stronger sense: for every sequence $h_k\to   0$, the convergence holds for $\mathbb P$-a.e.  $\omega\in\Omega$.
\end{thm}

 Analogous results (and even easier, in many ways)  can be stated for the case $n<\infty 
$.

\begin{prop}\label{prop:finite-n} Fix $n\in \{2,3,\dots\}$
Consider a financial model for which  \eqref{NUPBR} holds, and which is in $\mathcal{SD}(n)$.
Let $U\in \mathcal{CMIM} (n)$  so $V=V_U\in \mathcal D(n)\cap \mathcal C$. Furthermore, assume \eqref{finDual} and that  $V$ satisfies the inequalities
\begin{equation}\label{eq:rt-bounds-n-V}
0<c_k \leq -\frac{y \ V ^{(k+1)}(y)}{ V ^{(k)}(y)}=  \frac{(-y)^{k+1}  V ^{(k+1)}(y)}{ (-y)^{k} V ^{(k)}(y)} \leq d_k<\infty,\quad y>0,~ k=1,\dots, n-1, 
\end{equation} for some constants $c_k, d_k$, $k=1,\dots, n-1$. 
Then, for  the optimal investment problem  \eqref{primalProblem}, the following assertions hold:
\begin{enumerate}\item 
The value function $u\in \mathcal{CMIM} (n)$. 
\item Up to order $n$, the derivatives of  the dual value function $v$ are given by 
\eqref{derv}.
\item The dual value function  $v$ satisfies the  bounds   \eqref{eq:rt-bounds-n-V} with respect to the same constants $c_k$ and $d_k$,  for $k=1,\dots, n-1$.
\end{enumerate}

\end{prop}

 As mentioned earlier, stochastic dominance of  order  $n\in \{3,  \dots\}\cup \{\infty\}$ is a weaker condition than second order stochastic dominance ($n=2$). Therefore, the  class of models  considered here  is more general than the market models with a maximal element in the  sense of second order dominance in \cite{KS2006b}.  
 On the other hand, for $U\in \mathcal {CMIM}(n)$ and  an $\mathcal{SD}(n)$ model,  Lemma \ref{lem:same-dual} herein and  \cite[Theorem 5]{KS2006b} show that {\it the risk tolerance wealth process} $R(x)$ exists, for every $x>0$.
 
We briefly recall the definition of $R(x)$ in 
  \cite{KS2006b}  as the  maximal wealth process $R(x) = (R_t(x))_{t\in[0,T]}$, such that $R_T(x) = -\frac{U'(\widehat X_T(x))}{U''(\widehat X_T(x))}.$
The process $R(x)$ is equal, up to an initial value, to the   first derivative of the primal optimizer:
$$\lim\limits_{h\to0}\frac{\widehat X_T(x +h) - \widehat X_T(x)}{h}  = \frac{R_T(x)}{R_0(x)},$$
where the limit is in $\mathbb P$ probability.  

We also note that the existence of the   risk tolerance wealth process  is  connected to asymptotic expansions  of second order (see  \cite{KS2006b},  \cite{KS2007}) or
 \cite{MostovyiSirbuModel}). 
 For a formulation and asymptotics with consumption and their relationship to the risk tolerance wealth process, we refer to  \cite{HerdJohannes}. 



\begin{rem}[On an integrability convention]\label{remInt} Recalling a common convention (see, e.g., \cite{Mostovyi2015}), we set 
\begin{equation}\label{751}\mathbb E\left[ V(Y_T)\right] :=  \infty \quad if\quad \mathbb E\left[ V^+(Y_T)\right] =\infty.
\end{equation}
This is done to avoid issues related to $\mathbb E\left[ V(Y_T)\right]$  not being well-defined. 

However, the following argument shows that $\mathbb E\left[ V(Y_T)\right]$ is well-defined for every $y>0$ and $Y\in\mathcal Y(y)$. To see this, take 
an arbitrary $y>0$ and $Y\in\mathcal Y(y)$. Then, from the  conjugacy between $U$ and $V$, we get
$$-V(Y_T) \leq Y_T - U(1)\leq Y_T + |U(1)|.$$
Therefore,  by the supermartingale property of $Y$, we  obtain  $$\mathbb E\left[ V^-(Y_T)\right]\leq y +  |U(1)|<\infty.$$ Therefore, $\mathbb E\left[V(Y_T)\right]$ is well-defined with or without convention \eqref{751}. 

In analogy for the primal problem, we set
\begin{equation}\label{755}
\mathbb E\left[U(X_T)\right]:=-\infty\quad if\quad \mathbb E\left[U^-(X_T) \right] = \infty.
\end{equation}
 Assume that \begin{equation}\label{756}v(y)<\infty\quad for~some~y>0,
\end{equation} which is even weaker than \eqref{finDual}. 

By the argument from the previous paragraph,  we know  that $\mathbb E\left[V(Y_T)\right]\in\mathbb R.$  
For an arbitrary $x>0$ and $X\in\mathcal X(x)$, by conjugacy  between $U$ and $V$, we get 
$$U(X_T) \leq V(Y_T) + X_TY_T\leq  V^+(Y_T) + V^{-}(Y_T) + X_TY_T.$$
Therefore, $$U^+(X_T) \leq V^+(Y_T) + V^{-}(Y_T) + X_TY_T~~\in~~\mathbb L^1(\mathbb P).$$
Thus, $\mathbb E\left[ U(X_T)\right]$ is well-defined for every $x>0$, and $X\in\mathcal X(x)$ with  or without  the convention \eqref{755},  under the minimal assumption \eqref{756}.
\end{rem}

 \begin{rem}[On the positivity of $\widehat X_T^{(1)}$] We note that, in general, the derivative of the primal optimizer with respect to the initial wealth does not have to be a positive random variable, see \cite[Example 4]{KS2006}. However, for models satisfying the assumptions of Theorem \ref{thm1},  
\eqref{derX} implies the strict positivity of $\widehat X_T^{(1)}$. This complements the results in \cite{KS2006b}, see Theorem 4 there, which implies the positivity of $\widehat X_T^{(1)}$ in stochastically dominant models in the sense of the second order stochastic dominance. 
\end{rem}

\begin{proof}[Proof of Theorem \ref{thm1}]
As $U\in\mathcal {CMIM}$, its convex conjugate $V\in \mathcal D\cap C$. Then, Lemma \ref{lem:same-dual} gives that 
 $$v(y) = \mathbb E\left[V(y\widehat Y_T)\right],\ y>0,$$ i.e.,  the dual minimizer is the same up to the multiplicative constant $y$:
  $\widehat Y(y) = y\widehat Y$, $ y>0$.  
The expectation  above is well-defined, see the  discussion  in  Remark  \ref{remInt}. Note that \eqref{NUPBR}, \eqref{finDual} and the structure of the utility function (the Inada conditions, together with the strict monotonicity, concavity, and smoothness) imply the strict concavity and continuous differentiability of both $u$ and $-v$ on $(0,\infty)$; see \cite{KS2003}, \cite{Mostovyi2015}, and \cite{MostovyiNUPBR}. 
For $n= 1$, \eqref{derv} follows from the standard conclusions of the utility maximization theory, as
\begin{equation}\label{781}
-v'(y)y =-\mathbb E\left[V'(\widehat Y_T(y))\widehat Y_T(y) \right]=-\mathbb E\left[V'(y\widehat Y_T)y\widehat Y_T \right]\in(0,\infty),
\end{equation} 
and therefore, $v'(y)  = \mathbb E\left[V'(y\widehat Y_T)\widehat Y_T \right]$, for every $y>0$. 

Next, assume that \eqref{derv} holds for $n=k$, i.e., 
$$v^{(k)}(y)  =  \mathbb E\left[ V^{(k)}\left(y\widehat  Y_T\right)\left(\widehat Y_T\right)^k\right]=  (-1)^k\mathbb E\left[ \int\limits_0^\infty e^{-y\widehat  Y_Tz }z^{k-1}\mu(dz)\left(\widehat Y_T\right)^k\right],\quad y>0,$$
where in  the second equality we have used
 \eqref{bernstein}. Let us recall that \eqref{bernstein} gives
 \begin{displaymath}
 V^{(k)}(y)  = (-1)^k\int\limits_0^\infty e^{-yz}z^{k-1}\mu(dz),\quad y>0.
 \end{displaymath}
 Then, let us consider 
 \begin{equation}\label{763}
 \begin{split}
 &\frac{v^{(k)}(y+h) -v^{(k)}(y)}{h} =
 \frac{1}{h}\mathbb E\left[V^{(k)}((y+h)\widehat Y_T)\widehat Y_T^k - V^{(k)}(y\widehat Y_T)\widehat Y_T^k  \right] \\
 = &(-1)^k \mathbb E\left[\int\limits_0^\infty \frac{(\widehat Y_Tz)^{k-1}\widehat Y_T}{h}\left(e^{-(y+h)\widehat Y_Tz} - e^{-y\widehat Y_Tz}\right)\mu(dz)\right],
 \end{split}
 \end{equation}
 Let us fix $y> 0$. As for every $h\neq 0$, we have 
 $$
 0\leq -\frac{1}{h}\left(e^{-(y+h)\widehat Y_Tz}- e^{-y\widehat Y_Tz}\right)\leq e^{-(y-|h|)\widehat Y_Tz}z\widehat Y_T,$$
we deduce that, for a constant $h_0\in(0,y)$, and every $h\in(-h_0, h_0)$, the following inequalities hold 
   \begin{equation}\label{764}
  \begin{split}
 0\leq& -(\widehat Y_Tz)^{k-1}\frac{1}{h}\left(e^{-(y+h)\widehat Y_Tz} - e^{-y\widehat Y_Tz}\right)\widehat Y_T
 \leq (\widehat Y_Tz)^{k} e^{-\widehat Y_T(y-h_0)z}\widehat Y_T.
 \end{split}
 \end{equation}
Furthermore, as there exists a constant $M$, such that 
 \begin{equation}\label{7111}
 \bar z^{k}e^{-\bar z(y-h_0)}\leq M e^{-\tfrac{1}{2}\bar z(y-h_0)},\quad for~every ~\bar z\geq 0,
 \end{equation} and since,  by \eqref{derv} for $n=1$ (see  also \eqref{781}), we have 
 \begin{equation}\label{7112}\mathbb E\left[\int\limits_0^\infty M\widehat Y_Te^{-\tfrac{1}{2}\widehat Y_T(y-h_0)z} \mu(dz)\right]= -Mv'\left(\tfrac{1}{2}(y-h_0)\right)<\infty,
 \end{equation}
we deduce from \eqref{7111} and \eqref{7112} that the last expression in \eqref{764} is $\mathbb P\times \mu$ integrable.
 Therefore, in \eqref{763}, one can pass to the limit as $h\to 0$ to deduce that \eqref{derv} holds for $n  = k+1$. We conclude that \eqref{derv} holds for every $n\in\mathbb N$. Now, the complete monotonicity of $v$ follows from the complete monotonicity of $V$ and \eqref{derv}. In turn, this implies the analyticity of $v$, see, e.g., \cite{cmAnaliticity}. Further, as $(-1)^nV^{(n)}$ do not vanish (see, e.g., \cite[Remark 1.5]{BookBernsteinFunctions}), we deduce from \eqref{derv} that $(-1)^nv^{(n)}$ are also strictly positive for every $n\in\mathbb N$.  By  \cite[Theorem 4]{KS2003}, $-v$ satisfies the Inada conditions, which imply that $\nu(\{0\}) = 0$ and $\nu((0,\infty)) = \infty$.  Representation \eqref{bernRepv} follows, where \eqref{propertyNu} results from the 
 inversion formula, see \cite[Chapter VII,
Theorem 7a]{Widder41}.

To obtain the properties of $u$, first, we observe that the biconjugacy relations between the  value functions imply that $u'$ exists at every $x>0$, and it is the inverse of $-v'$. This, and since $v'$ is strictly negative on $(0,\infty)$, imply the analyticity of $u$. 
 In turn, \eqref{deru} is the consequence of the Fa\`a di Bruno formula, see \cite[Section 4.3]{Faa}. 
 \end{proof}

\begin{proof}[Proof of Theorem \ref{thm2}]
First, we observe that \eqref{7113} is the consequence of  \eqref{bernstein} and standard assertions of the utility maximization theory. In turn, \eqref{derY} follows from  the optimality of $y\widehat Y$  for every $y>0$, whereas  \eqref{derX} results from the relation 
 $$\widehat X_T(x) = -V'(y\widehat Y_T),\quad{\rm for}\quad y=u'(x).$$ 
 
 The higher order derivatives of the dual and primal optimizers follow from the direct computations and an application  of the Fa\`a di Bruno formula.
\end{proof}

\begin{proof}[Proof of Proposition \ref{prop:finite-n}] From Lemma \ref{lem:same-dual}, we have 
$$v(y)=\mathbb E[V(y\widehat Y_T)], \quad y>0.$$ By \cite[Theorem 4]{KS2003}, $v$ is differentiable and we have
$$v'(y)=\mathbb E[\widehat Y_T V'(y\widehat Y_T)], \quad y>0.$$
Therefore, it only remains to compute the higher order derivatives of $v$, recursively, up to order $n$, as in formula \eqref{derv}. The bounds \eqref{eq:rt-bounds-n-V} for $v$ would follow immediately. In what follows, we show that \eqref{derv} holds up to order $n$.  Assume that, for some $1\leq k\leq n-1$, we have 
\begin{equation}
\label{induction:k-1}(-1)^{k-1}v^{(k-1)}(y)= \mathbb E[(-\widehat Y_T)^{k-1} V^{(k-1)}(y\widehat Y_T)]<\infty, \quad y>0.\end{equation}
Using bounds \eqref{eq:rt-bounds-n-V} and following the proof of \cite[Lemma 3]{KS2006} we obtain that, for any $a>1$, there exist some constants 
$$\alpha _k <1< \beta _k<\infty$$ for which 
\begin{equation}\label{alpha-beta}\alpha _k (-1)^{k-1}V^{(k-1)} (y)\leq (-1)^{k-1}V^{(k-1)} (ay) \leq \beta _k (-1)^{k-1}V^{(k-1)} (y), \quad   y>0.
\end{equation}
Fix $y>0$. 
Then, 
\begin{equation*}\begin{split} \frac{v^{(k-1)}(y+h)-  v^{(k-1)}(y)}h & =\mathbb E\left [(\widehat Y_T)^{k-1} \frac{V^{(k-1)}((y+h)\widehat Y_T)  -V^{(k-1)}(y\widehat Y_T) }h\right] \\ & = \mathbb E\left [(\widehat Y_T)^k V^{(k)}( \xi_h) \right],
\end{split}
\end{equation*}
where $\xi _h$ is a random variable taking values between $y\widehat Y_T$ and $(y+h)\widehat Y_T$. 

Fix $a>1$. Using the bounds \eqref{eq:rt-bounds-n-V} for $V^{(k)}$ in terms of $V^{(k-1)}$,  together with \eqref{alpha-beta}, we conclude that there exists a finite constant $C$, such that for 
 $|h|$  small enough so that
$$ \frac 1a \leq \frac{y-|h|}y \leq \frac{y+|h|}y\leq a,$$ we have
$$\left | (\widehat Y_T)^k V^{(k)}( \xi_h) \right| \leq C \left | (\widehat Y_T)^{k-1} V^{(k-1)}( y\widehat Y_T) \right|.$$
Since the right-hand side above is integrable, according to   \eqref{induction:k-1}, and 
$$\frac{V^{(k-1)}((y+h)\widehat Y_T)  -V^{(k-1)}(y\widehat Y_T) }h\rightarrow \widehat Y_T V^{(k)}(y\widehat Y_T), \quad \mathbb P-a.s,$$
we can use the  Lebesgue dominated convergence theorem to conclude the assertions of part $ii)$. The remaining assertions follow.
\end{proof}
\subsection{\textcolor{black}{An Example}} \textcolor{black}{ The  condition that $\mathcal{Y}$(1) has a maximal element means that, while the market is incomplete with respect to the replication of contingent claims, it behaves like a complete market from the point of view of optimal investment.  Our main result, Theorem \ref{thm1}, may appear restrictive at first, for this reason. However, the counter-example in Section \ref{secExample2} shows that this is actually the best one can hope.   A generic example has to be precisely such a market model with a maximal dual element.  This is the case for a multi-dimensional market driven by Brownian motion with a larger dimension, where the coefficients (market prices of risk) are deterministic. For simplicity, we present a one-dimensional stock driven by two Brownian motions. }

\textcolor{black}{
Let $W^1 = (W^1_t)_{t\in[0,T]}$ and $W^2 = (W^2_t)_{t\in[0,T]}$ be two independent Brownian motions on a complete stochastic basis $(\Omega, \mathcal F,  (\mathcal F_t)_{t\in[0,T]},  \mathbb P)$, where the filtration $(\mathcal F_t)_{t\in[0,T]}$ is generated by $W^1$ and $W^2$. Let us suppose that  there are two traded securities: a riskless asset with zero interest rate, $B_t = 1$, $t\in[0,T]$, and a traded stock, whose dynamics is given by
\begin{displaymath}
{S_t} = S_0 + \int\limits_0^t \mu_s ds + \int\limits_0^t\sigma_s dW^1_s,\quad t\in[0,T].
\end{displaymath}
for some $S_0\in\mathbb R$, and where $\mu$ and $\sigma$ are {\it deterministic} measurable functions on $[0,T]$, such that, for some constant $M>0$,  $-M\leq \mu_t\leq M$ and $\frac 1M \leq\sigma_t\leq M$, $t\in[0,T]$. 
}

\textcolor{black}{
Next,  consider the process $\widehat Y$ given by 
\begin{equation}\label{hatYexa}
 \widehat Y_t = \exp\left( -\int\limits_0^t \frac{\mu_s}{\sigma_s}dW^1_s -\frac 12\int\limits_0^t \frac{\mu^2_s}{\sigma^2_s}ds\right),\quad t\in[0,T]. 
\end{equation}
Following the argument in \cite[Example 7, p. 2185]{KS2006b}, one can show that 
\begin{equation}\label{9223}
\widehat Y_T {\succeq}_2 Y_T,\quad {\rm for~every\quad}Y\in\mathcal Y(1),
\end{equation}
in the sense of Definition \ref{def:dominance-fin}, where we recall that 
the set of supermartingale deflators $\mathcal Y(1)$ is given in \eqref{defY}.
}

\textcolor{black}{The reader may note that the market $(B, S, (\mathcal{F}^S_t)_{0\leq t\leq T})$ is complete, while the original market  $\left(B, S, (\mathcal F_t)_{0\leq t\leq T}\right)$, with all the information available for investment, is incomplete.} 
\textcolor{black}{
In such a market (either one), let us consider the optimal investment problem  \eqref{primalProblem}. If  $U\in \mathcal{CMIM}$ and  \eqref{finDual} holds, then, with $\widehat Y$ being given by \eqref{hatYexa}, the properties of the primal and dual value functions are given by the assertions of Theorems \ref{thm1}, whereas the properties of the optimizers are given by the conclusions of Theorem \ref{thm2}. If, instead, one supposes that for a fixed $n\in \{2,3,\dots\}$,  $U\in \mathcal{CMIM} (n)$ and is such that its convex conjugate 
 $V$ satisfies the inequalities
\eqref{eq:rt-bounds-n-V}
for some constants $c_k$, $d_k$, $k=1,\dots, n-1$, then, under \eqref{finDual}, we obtain from Proposition \ref{prop:finite-n}  that its assertions  apply.  
}
%
\section{counterexample 1: $\mathcal{SD}\left( \infty \right) $ market
model and $U\protect \notin \mathcal{CMIM}$ }

\label{exampleMihai} We show that the analyticity of the value function may  
fail if the utility is not $\mathcal{CMIM}$, even if it is analytic, and
even if the market model is complete, and thus in the $\mathcal{SD}%
\left( n\right) $ class for every  $n\in \left \{ 2,3,...\right \} \cup \{ \infty \}$. 
As the construction shows, we will be using completely monotonic functions
of finite order. Working with this class allows to tailor the assumptions on
the utility function so that differentiability holds up to order $n$, but
fails at order $n+1$, for \textit{any} choice of $n\in\mathbb N$.

\begin{prop}\label{prop:example} Fix $n\geq 1$. There exists a complete market model and an analytic utility function $U:(0, \infty)\rightarrow \mathbb R$  such that  $$U\in \mathcal{CMIM}(n+1) \Longleftrightarrow V\in \mathcal D(n+1)\cap \mathcal C,$$ where the dual $V$  satisfies  the bounds \eqref{eq:rt-bounds-n-V} (up to order  $k=n-1$, but not up to order $k=n$), and for which the  conjugate value functions  $u$ and $v$ satisfy
$$u\in \mathcal{CMIM} (n)  \Longleftrightarrow v\in\mathcal D(n)\cap C,$$
 together with identical bounds  \eqref{eq:rt-bounds-n-V} up to order $k=n-1$,  but with
$$(-1)^{n+1} v^{(n+1)}(1)= \infty. $$
\end{prop}


We will need the following Lemma.

\begin{lem}\label{lemma:analytic} There exists an analytic function $f:(0,\infty)\rightarrow \mathbb{R}$ with the following properties
\begin{enumerate}
\item $1\leq f\leq 2$, 
\item $f'<0$,
\item $-f'(i)\geq  \frac {i^2}C, \ i=1,2,\dots$, for some  constant $C>0$.
\end{enumerate}\end{lem} 

\begin{proof}[Proof of Proposition \ref{prop:example}] Consider the auxiliary dual value function $\bar V(y):=y^{-1}$.
Consider now the new utility function $V$ defined as \begin{equation}\nonumber
V (y):=\int\limits_y ^{\infty}  \int\limits_{y_1}^{\infty} \dots \int\limits_{y_{n-1}}^{\infty} (-1)^n f(y_n)\bar V^{(n)}(y_n)    dy_n \dots  dy_2  dy_1>0.
\end{equation}
Note that 
the intuition behind this definition comes from setting $V$ at the level of the $n$-th order derivative,  
$$ V ^{(n)}(y):=f(y)\bar V ^{(n)}(y), \quad y>0,$$
and then recovering $V$ by integration.

Since $1\leq f\leq 2$  and using the integral representations
$$ (-1)^k \bar V^{(k)}(y_k)= \int\limits_{y_k}^{\infty} \dots ~
\int\limits_{y_{n-1}}^{\infty} (-1)^n \bar V^{(n)}(y_n)    dy_n 
~\dots  dy_{k+1},\quad y_k>0,$$
$$ (-1)^k  V^{(k)}(y_k)= \int\limits_{y_k}^{\infty} \dots ~
\int\limits_{y_{n-1}}^{\infty} (-1)^n f(y_n) \bar V^{(n)}(y_n)    dy_n 
~\dots  dy_{k+1},\quad y_k>0,$$
we obtain  bounds for derivatives of $V$ in terms of derivatives of the same order of $\bar V$,
\begin{equation}\label{eq:Vk} (-1)^k \bar V^{(k)}(y)\leq  (-1)^k  V^{(k)}(y) \leq 2  (-1)^k \bar V^{(k)}(y)\quad y>0,~k=0,1,\dots, n.
\end{equation}
Since  $$ \bar V''(y)=2y^{-3}, \bar V'''(y)=-6y^{-4},\dots , \bar V^{(n+1)}(y)= (-1)^{n+1} C_{n+1} y^{-n-2},\quad y>0,$$
for some  explicit positive  constants $C_k$,  $\bar V$ satisfies   some  bounds on higher order risk tolerance coefficients  of the type \eqref{eq:rt-bounds-n-V} up to order $k=n-1$. Using the integral representations above, 
\eqref{eq:Vk} yields similar bounds for 
 risk tolerance type coefficients for $ V$: 
for every $ k=1,\dots, n-1$, there exist $0<c_k < d_k < \infty$
such that 
\begin{equation}\nonumber
c_k \leq -\frac{y V ^{(k+1)}(y)}{ V ^{(k)}(y)}=  \frac{(-y)^{k+1}  V ^{(k+1)}(y)}{ (-y)^{k} V ^{(k)}(y)} \leq d_k,\quad y>0,~ k=1,\dots, n-1.
\end{equation}

Since $-f'(i)\geq \frac{i^2}{C}$, we  choose weights $q_i:=\frac 1{i^3}>0$, such that
$$s_1:=\sum _{i=1}^{\infty} i q_i<\infty,\ \ \sum _{i=1}^{\infty} -f'(i)q_i=\infty.$$ Denoting by
$$s_0:= \sum _{i=1}^{\infty}  q_i < \sum _{i=1}^{\infty} iq_i=s_1<\infty,$$  we define 
$$\varepsilon :=\frac{\frac 12}{\frac{s_1}{s_0}-\frac 12}\in (0,1)$$
and  consider a random variable $Z$ such that 
 $$\mathbb P(Z=\frac 12)=1-\varepsilon, \  \mathbb P (Z=i)=p_i:=\varepsilon\frac {q_i}{s_0}, \ \ i=1, 2, 3,  \dots.$$
Then, we have  $$\mathbb E[Z]= \frac 12 \left(1- \varepsilon\right) +\frac{\varepsilon}{s_0} s_1 =\frac 12 +\varepsilon \left(\frac{s_1}{s_0}-\frac 12\right)=1.$$ 
Next, consider any market with the unique martingale measure with density $Z$. 
The dual value function is finite since $Z$ is bounded from below. Therefore, 
\begin{equation} v (y)=\mathbb E[  V (yZ))]<\infty, \quad y>0,
\end{equation}
i.e., \eqref{finDual} holds. 
Using Proposition  \ref{prop:finite-n}, we  obtain that 
\begin{equation}\label{eq:bfVk}
0< (-1)^k v ^{(k)}(y)=\ (-1)^k \mathbb E[  V ^{(k)} (yZ)Z^k)]<\infty, \quad y>0, \ k=1,\dots, n.
\end{equation}
Therefore, $v\in \mathcal D(n)\cap \mathcal C$ and, in particular,  the $n$-th derivative is given by
\begin{equation}\label{eq:vn}
 v ^{(n)}(y)= \mathbb E[   V ^{(n)} (yZ)Z^n)], \quad y>0.
\end{equation}
The $n+1$ derivative of $V$ is
$$V^{(n+1)}(y)=f(y)\bar V^{(n+1)}(y)+f'(y)\bar V^{(n)}(y).$$
One should note that
$$(-1)^{n+1}V^{(n+1)}(y)= f(y) (-1)^{n+1}\bar V^{(n+1)}(y) +(-f'(y) ) (-1)^n \bar V^{(n)}(y)>0,$$
and 
$$(-y)^{(n+1)}   V^{(n+1)}(y)> -C_n f'(y).$$  By construction, $\mathbb E[-f'(Z)]=\infty,$ so 
$\mathbb E[(-Z)^{n+1} V^{(n+1)}(Z)]=\infty.$ 
Finally, from \eqref{eq:vn} we have 
$$  (-1)^{n+1}\frac {v^{(n)}(y)-v^{(n)}(1)}{y-1}  = (-1)^{n+1}\frac {\mathbb E[   V ^{(n)} (yZ)Z^n) -V ^{(n)} (Z)Z^n)]}{y-1}=\mathbb{E}[V^{(n+1)}(\xi) (-Z)^{n+1}],$$ for some random variable $\xi$ taking values between $Z$ and $yZ$. 
Since
$$0\leq V^{(n+1)}(\xi) (-Z)^{n+1}\rightarrow  V^{(n+1)}(Z) (-Z)^{n+1},$$
 we  can now apply Fatou's lemma to obtain 
$$(-1)^{n+1}v^{(n+1)}(1)= \infty. $$
\end{proof}

\begin{proof}[Proof of Lemma \ref{lemma:analytic}]
Consider the Gaussian densities (up to a multiplicative factor) with mean $\mu$ and standard deviation  $\sigma>0$, given by
$$g_{\mu, \sigma }(x):=\frac{1}{\sigma} e^{-\frac {(x-\mu)^2}{2\sigma ^2}},\quad x\in \mathbb R.$$
Then,  we have 
$$\int\limits_{-\infty}^{\infty}  g_{\mu, \sigma }(x)dx=\sqrt {2\pi},  \quad g_{\mu,\sigma ^2}(\mu)=\frac 1{\sigma}.$$
Next, let 
$$g(z):=\sum _{i=1}^{\infty} \frac 1{i^2}   g_{i, i^{-4}} (z),\quad z \in \mathbb C.$$
While not obvious, it is easy to see that the series converges uniformly  on compacts (in the complex plane), so it is analytic (entire). In addition, 
$$\int\limits_{-\infty}^{\infty}  g(x)dx \leq \sqrt{2\pi}  \sum _{i=1}^{\infty} \frac 1{i^2}=:C<\infty.$$ Furthermore, we have
$$g(i)\geq  \frac 1{i^2}   g_{i, i^{-4}}(i)=   \frac 1{i^2} i^4=i^2.$$
Finally, we set 
$$f(y):=2-\frac 1 C \int\limits_0^y g(x)dx, \quad y>0.$$
One can then see that $f$ satisfies the desired properties.
\end{proof}

\begin{rem} To obtain an example with a  utility  having positive third derivative  $ U{'''}>0$,   we just  use   Proposition \ref{prop:example} for  $n=2$.   The fact that $ -V'''= (-1)^{(n+1)}V^{(n+1)}>0$  ensures, by duality, that the corresponding utility 
$$ U(x):=\inf _{y>0}\left (V(y)+xy\right), \quad x>0,$$
satisfies the desired condition
$ U'''(x)>0$ for all $  x>0.$
\end{rem}
\section{Counterexample 2: non-stochastically dominant models and lack of differentiability}\label{secExample2}
 We show that for \textit{any} non-homothetic utility $%
U\in \mathcal{-C}$ with $U\in C^{2}\left( \left( 0,\infty \right) \right)$, we may construct a non-$\mathcal{SD}\left( \infty \right) $ market model
such that, at some point $x>0$, the two-times differentiability of $u$ fails. 
We recall that standard results in utility maximization theory, in the form
of Kramkov and Schachermayer \cite{KS}, assert 
the continuous differentiability of the value functions. The result below
demonstrates that differentiability might cease to exist at the very next
order (even with a $\mathcal{CMIM}$ utility). 

We note that, due to the multiplicative structure $\widehat{Y}(y)=y\widehat{Y}$ under the assumptions of Theorem \ref{thm1}, we do not
make any sigma-boundedness assumption, as in \cite{KS2006}. Our counterexample  is somewhat related to the sigma-boundedness counterexample from \cite{KS2006}, but it  is stronger: we construct a (counterexample) model for  {\it every} Inada utility function with non-constant relative risk aversion.

Let $U\in \mathcal{-C}$ with $U\in C^{2}\left( \left( 0,\infty \right)
\right) ,$ having non-constant relative risk aversion 
\begin{equation}\label{A-function}
A(x):=-\frac{U^{\prime \prime }(x)x}{U^{\prime }(x)},x>0.  
\end{equation}
The assumption $U\in C^{2}\left( \left( 0,\infty \right) \right) $ is
without loss of generality. We may also choose $U\in \mathcal{CMIM}.$

\begin{prop}For any non-homothetic\footnote{I.e., such that $A\neq const$, where $A$ is defined in \eqref{A-function}.} utility $U\in \mathcal{-C},$
with $U\in C^{2}\left( \left( 0,\infty \right) \right)$, and thus,  $U\in\mathcal {CMIM}(2)$, there exists a non-%
$\mathcal{SD}\left( \infty \right) $ market model such that the value
function is not twice differentiable at some $x>0$.
\end{prop}

\begin{proof}
We first assume that the risk aversion $A$ satisfies  $A(1/m)\neq A(1/k)$, for some $m$ and $k$ in $\mathbb N$.
As we justify at the end of the proof, this is without loss of generality. 

Let us suppose that the sample space $\Omega = \{\omega_0, \omega_1,\dots\}$, and consider a one-period  model, where the market consists of a money market account with $0$ interest rate and a stock, with $S_0 = 1$ and $S_1(\omega_0) = 2$, $S_1(\omega_n) = \frac{1}{n}$, $n\in\mathbb N.$


We are going to construct  probabilities $p_n :=\mathbb P[\omega_n]>0$, $n\geq 0$,  satisfying the following three properties
\begin{equation}\label{finite-needed}
\mathbb E\left[-U''(S_1)\right]<\infty,
\end{equation}
\begin{equation}\label{6281}
\mathbb E\left[U'(S_1) S_1\right] = \mathbb E\left[U'(S_1)\right]<\infty,
\end{equation}
and
\begin{equation} \label{non-constant-A}
\mathbb E\left [U'(S_1)(1-S_1) A(S_1)\right]\not =0,
\end{equation}
where $A$ is defined in \eqref{A-function}. Note that, relations \eqref{6281} and \eqref{non-constant-A} can hold together only if the function $A$ is non-constant.

Direct computations show  that  \eqref{non-constant-A} holds if and only if
\begin{equation}\label{6282}
\widehat \Delta:=-\frac{\mathbb E\left[ U''(S_1)(S_1 -1)\right]}{\mathbb E\left[ U''(S_1)(S_1 -1)^2\right]}\neq1.
\end{equation}
Furthermore, note that $\widehat \Delta\in(-1,2)$. 

In addition to \eqref{finite-needed}, \eqref{6281}, and \eqref{non-constant-A}, we will show that there exists $\bar\varepsilon\in\left(0,\frac12\right]$, such that for a random variable defined as 
\begin{displaymath}
\begin{split}
G(\omega):=&
\min\limits_{\varepsilon\in[0, \bar \varepsilon]}U''\left(S_1(\omega)+\varepsilon(1 +\widehat \Delta (S_1(\omega) - 1))\right)1_{ \{\widehat \Delta <1\}}\\
&+\min\limits_{\varepsilon\in[-\bar\varepsilon, 0]}U''\left(S_1(\omega)+\varepsilon(1 +\widehat \Delta (S_1(\omega) - 1))\right)1_{ \{\widehat \Delta>1\}},\quad \omega\in\Omega,
\end{split}
\end{displaymath} satisfies
\begin{equation}\label{711}
G\in\mathbb L^1(\mathbb P).
\end{equation}

Assuming for now that such probabilities indeed exist, we show that under \eqref{6281} and stock as above, $\mathbb P$ is not a martingale measure for $S$, and we have
\begin{equation}\label{712}
2>\mathbb E\left[S_1\right]> 1.
\end{equation}
Indeed, the monotonicity of $U'$ yields
\begin{displaymath}
\begin{split}
&\mathbb E\left[1-S_1\right] = \mathbb E\left[(1- S_1)1_{\{S_1<2\}} \right] - \mathbb P[\{S_1 = 2\}] \\
<&\mathbb E\left[ \frac{U'(S_1)}{U'(2)}(1 -S_1)1_{\{S_1<2\}} \right] - \mathbb P[\{S_1 = 2\}]\\
=&\frac{1}{U'(2)}\left(\mathbb E\left[ U'(S_1)(1 -S_1)1_{\{S_1<2\}} \right] + {U'(2)}\mathbb E[(1-S_1)1_{\{S_1 = 2\}}]\right) \\
=& \frac{1}{U'(2)}\mathbb E\left[ U'(S_1)(1 -S_1)\right] = 0,
\end{split}
\end{displaymath} 
where in the last equality we used \eqref{6281}. This implies 
\eqref{712}, where the upper bound is also strict as $p_n>0$, for every $n\geq 0$. Thus, $\mathbb P$ is not a martingale measure for $S$. Therefore, the constant-valued process $Z\equiv 1$ is not an element of $\mathcal Y(1)$, and thus it is not the dual minimizer for $y=1$.

Furthermore, we claim that \eqref{finDual} holds. 
This is rather clear by observing that, for $n_0$ large enough, one can choose a martingale measure $\mathbb Q$  that changes the probabilities only for $\omega _0, \omega _1, \dots, \omega _{n_0}$ for some $n_0$, but keeps the same  probabilities for $\omega _n, n>n_0$. As the density  $Z\in \mathcal{Y}(1)$ of such a martingale measure is bounded below away from $0$, \eqref{finDual} holds. 

Next, we construct appropriate probabilities $p_n$'s (such that  \eqref{6281},  \eqref{711}, and \eqref{6282} hold; 
note that one might need to perturb finitely many of these $p_n$'s later such that \eqref{8301} below holds too). For this, we set
\begin{equation}\label{725}
p_n := \frac{1}{2^{n+1}}\frac{\min(1, U'(2))}{\max\left(1,U'\left(\tfrac{1}{n}\right) -\min\limits_{s\in\left[\frac{2}{3n},\frac{2}{3} + \frac{2}{3n}\right] }U''\left(s\right)\right)},\quad n\geq  2.
\end{equation}
Note that $$\min\limits_{s\in\left[\frac{2}{3n},\frac{2}{3} + \frac{2}{3n}\right] }U''\left(s\right)\leq \min\limits_{z\in[0,1/3]}\min\limits_{\Delta \in[-1,1]}U''(1/n  + z(1 +\Delta (1/n-1)))$$ 
and $$\min\limits_{s\in\left[\frac{2}{3n},\frac{2}{3} + \frac{2}{3n}\right] }U''\left(s\right)\leq\min\limits_{z\in[-1/3,0]}\min\limits_{\Delta \in[1,2]}U''(1/n  + z(1 +\Delta (1/n-1))).$$ The intuition behind the exact form of the intervals above comes from taking $\bar  \varepsilon = 1/3$ in the construction of $G$ satisfying \eqref{711} when  $\widehat \Delta$ is not fixed yet.

Then, as $S_1>1$ only for $\omega_0$, we~have 
\begin{displaymath}
\begin{split}
&0\leq\mathbb E\left[U'(S_1)S_1 \right]\leq U'(2) + \mathbb E\left[U'(S_1) \right] \leq 2U'(2) + U'(1) \\
&+ \sum\limits_{n\geq 2}\frac{1}{2^{n+1}}\frac{U'\left(\frac{1}{n}\right)}{\max\left(U'\left(\tfrac{1}{n}\right) -\min\limits_{s\in\left[\frac{2}{3n},\frac{2}{3} + \frac{2}{3n}\right] }U''\left(s\right),1\right)}<\infty,
\end{split}
\end{displaymath}
and, the finiteness in \eqref{6281} holds (regardless of the choice of $p_0$ and $p_1$). 

Now, with $p_n$, $n\geq 2$, given by \eqref{725}, we show that we can simultaneously have  \eqref{6281} and 
\eqref{6282}. 
We define 
\begin{equation}\label{741}
\begin{split}
p_0 := &\frac{1}{U'(2)}\sum\limits_{n\geq 2} p_nU'\left(\tfrac{1}{n}\right)\left(1 - \tfrac{1}{n}\right)\\
= &
\frac{\min(U'(2),1)}{U'(2)}
\sum\limits_{n\geq 2} 
\frac{1}{2^{n+1}}\frac{U'\left(\tfrac{1}{n}\right)\left(1 - \tfrac{1}{n}\right)}{\max\left(U'\left(\tfrac{1}{n}\right) -\min\limits_{s\in\left[\frac{2}{3n},\frac{2}{3} + \frac{2}{3n}\right] }U''\left(s\right)  ,1\right)}
 .\\
 \end{split}
\end{equation}
Then, using the above, we rewrite 
\begin{displaymath}
2p_0 U'(2) + \sum\limits_{n\geq 1} p_nU'\left(\tfrac{1}{n}\right)\frac{1}{n} = p_0 U'(2) + \sum\limits_{n\geq 1} p_nU'\left(\tfrac{1}{n}\right),
\end{displaymath}
and \eqref{6281} follows. Thus, \eqref{6281} holds with $p_n$, $n\geq 2$, given by \eqref{725} and $p_0$ specified by  \eqref{741}.
Note that $p_0\leq 1/4$, $\sum\limits_{n\geq 2}p_n\leq 1/4$, and, therefore, $p_1:= 1 - (p_0 +\sum\limits_{n\geq 2}p_n)\geq 1/2$.

To show \eqref{711}, we observe that there exist constants $a$ and $a'$, such that $0<a<a'$ and, for an appropriate $\bar \varepsilon$, we have 
\begin{displaymath}
\begin{split}
0\geq\mathbb E\left[G\right] \geq  &
 2\min\limits_{s\in\left[ a ,a'\right]}U''\left(s\right)  +
 \sum\limits_{n\geq 2}\frac{1}{2^{n+1}}\frac{\min\limits_{s\in\left[\frac{2}{3n},\frac{2}{3} + \frac{2}{3n}\right] }U''\left(s\right)}{\max\left(U'\left(\tfrac{1}{n}\right) -\min\limits_{s\in\left[\frac{2}{3n},\frac{2}{3} + \frac{2}{3n}\right] }U''\left(s\right)  ,1\right)}>-\infty.
\end{split}
\end{displaymath}

Next, we show that \eqref{non-constant-A} holds for the choice of $p_n$'s or for a slightly perturbed choice of $p_n$'s, where the distortion is such that the remaining assumptions of the example do not change. 
To this end, 
we rewrite \eqref{6282} as
\begin{equation}\label{745}
-U''(2) p_0 + \sum\limits_{n\geq 2}U''(\tfrac{1}{n})(1 - \tfrac{1}{n})p_n \neq U''(2) p_0 + \sum\limits_{n\geq 2}U''(\tfrac{1}{n})(1 - \tfrac{1}{n})^2p_n.
\end{equation}
Collecting terms and plugging  the expression for $p_0$ from \eqref{741}, we can rewrite  \eqref{745} as 
\begin{equation}\nonumber
0\neq  -\frac{2U''(2)}{U'(2)}\sum\limits_{n\geq 2} p_nU'\left(\tfrac{1}{n}\right)\left(1 - \tfrac{1}{n}\right) + \sum\limits_{n\geq 2}U''(\tfrac{1}{n})(1 - \tfrac{1}{n})\tfrac{1}{n}p_n
\end{equation}
or, in turn, 
\begin{equation}\label{792}
0\neq \sum\limits_{n\geq 2} p_nU'(\tfrac{1}{n})\left(1 - \tfrac{1}{n}\right)\left(A(2) - A(\tfrac{1}{n})\right).
\end{equation}


We note that, 
if $x\to A(x)$, $x>0$, is strictly monotone\footnote{An  example of an Inada utility function of class $\mathcal {CMIM}$, where the relative  risk aversion  is strictly monotone,  is given  via $-V'(y) = y^{-k}\frac{1}{y+1}$, $y>0$, for some constant $k>0$. Here $-V'$ is completely monotonic  as a product of the completely monotonic functions $y\to y^{-k}$ and $y\to\frac{1}{y+1}$, $y>0$, see \cite[Corollary 1.6]{BookBernsteinFunctions}. Then, the relative risk tolerance at $x=-V'(y)$ is given by $B(y) = -\frac{V''(y)y}{V'(y)} = k + \frac{y}{y+1}$, which is a strictly monotone function of $y$ on $(0,\infty)$. As $A(x)B(y) = 1$ for $y = U'(x)$, we deduce that $A$ is also a strictly monotone function on $(0,\infty)$.}, \eqref{6282} holds, as all terms  under the sum in \eqref{792} are non-zero and of the same sign.

When the relative risk aversion is not monotone, but also non-constant, and if \eqref{792} does not hold, it is enough to perturb finitely many of the $p_n$'s in a way to get simultaneously
\begin{equation}\label{8301}
\begin{split}
&\sum\limits_{n\geq 0} p_n = 1,\\
&\sum\limits_{n\geq 0} p_nU'(s_n) (1-s_n) = 0,\\
&\sum\limits_{n\geq 0} p_nU'(s_n) (1-s_n)A(s_n)\neq 0,
\end{split}
\end{equation}
while preserving the positivity of $p_n$'s (here $s_0 = 2$ and $s_n = 1/n$, $n\in\mathbb N$). As $A(1/m)\neq A(1/k)$, for some $m$ and $k$, such a distortion of $p_n$'s exists. As we have only perturbed finitely many $p_n$'s, \eqref{711} still holds. This results in the choice of a probability measure, such  that \eqref{6281}, \eqref{6282}, and \eqref{711} hold.

We show that $u''(1)$ does not exist. First, we will assume that in \eqref{6282}, the left-hand side is strictly less than $1$, i.e., 
\begin{equation}\label{724}
-\frac{\mathbb E\left[ U''(S_1)(S_1 -1)\right]}{\mathbb E\left[ U''(S_1)(S_1 -1)^2\right]}<1.
\end{equation}
As for every $x\geq 0$, $\{x + \Delta (S_1 - 1):~\Delta \in [-x,x]\}$ is the set of terminal values of the elements of $\mathcal X(x)$, we observe that buying the portfolio consisting of one share of stock is admissible for $x =1$. Then, by conjugacy, we have
\begin{equation}\label{6284}
\begin{split}
\mathbb E\left[ U(S_1) \right] = &~\mathbb E\left[ V(U'(S_1))+ U'(S_1)S_1\right].\\
\end{split}
\end{equation}
Using \eqref{6281} and with an arbitrary $\Delta \in[-1,1]$, we can rewrite the latter expression as 
\begin{equation}\label{6285}
\begin{split}
\mathbb E&\left[ V(U'(S_1))+ U'(S_1)S_1\right] = \mathbb E\left[ V(U'(S_1))+ U'(S_1)\right] \\
&= \mathbb E\left[ V(U'(S_1))+ U'(S_1)(1 + \Delta (S_1-1))\right] 
\geq \mathbb E\left[U(1 + \Delta (S_1-1))\right].
\end{split}
\end{equation}
Combining \eqref{6284} and \eqref{6285}, we get
\begin{equation}\nonumber
\mathbb E\left[ U(S_1) \right]  \geq \mathbb E\left[U(1 + \Delta (S_1-1))\right], \quad \Delta \in[-1,1],
\end{equation}
which yields that $\widehat  X(1) = S$. 

In turn, by the relations  between the primal and dual optimizers, we get
\begin{equation}\label{6291}
u'(1) = \mathbb  E\left[S_1U'(S_1)\right] = \mathbb  E\left[U'(S_1)\right],
\end{equation}
where the  second equality follows from \eqref{6281}. 

For $\varepsilon$ being small, Taylor's expansion yields
\begin{equation}\label{6292}
U(\widehat X_1(1 + \varepsilon)) - U(S_1) = \varepsilon U'(S_1) \frac{\widehat X_1(1 + \varepsilon) - S_1}{\varepsilon} + \frac{\varepsilon^2}{2}U''(\eta(\varepsilon))\frac{\left(\widehat X_1(1 + \varepsilon) - S_1\right)^2}{\varepsilon^2},
\end{equation}
where $\eta(\varepsilon)$ is a random  variable taking values between $S_1$ and $\widehat X_1(1+\varepsilon)$. Therefore, from \eqref{6291} and \eqref{6292}, we obtain
\begin{equation}\label{6293}
\begin{split}
\tfrac{2}{\varepsilon^2}\left(u(1 + \varepsilon) - u(1) - \varepsilon u'(1)\right)
= &~
\frac{2}{\varepsilon}\mathbb E\left[ U'(S_1) \left(\frac{\widehat X_1(1 + \varepsilon) - S_1}{\varepsilon} - 1\right)\right]
\\&+\mathbb E\left[ U''(\eta(\varepsilon))\frac{\left(\widehat X_1(1 + \varepsilon) - S_1\right)^2}{\varepsilon^2}\right].\\
\end{split}
\end{equation}
Let us consider the first term in the right-hand side of \eqref{6293},
\begin{equation}\label{6295}
\frac{2}{\varepsilon}\mathbb E\left[ U'(S_1) \left(\frac{\widehat X_1(1 + \varepsilon) - S_1}{\varepsilon} - 1\right)\right].
\end{equation}
As $\widehat X_1(1 + \varepsilon) = 1+\varepsilon + \widehat\Delta(1+\varepsilon) (S_1 - 1)$, for some (fixed and nonrandom) $\widehat\Delta(1+\varepsilon)\in[-1-\varepsilon, 1 + \varepsilon]$, we can rewrite \eqref{6295} as
\begin{equation}\label{6296}
\begin{split}
&\frac{2}{\varepsilon}\mathbb E\left[ U'(S_1) \left(\frac{1+\varepsilon + \widehat\Delta(1+\varepsilon) (S_1 - 1) - S_1}{\varepsilon} - 1\right)\right] \\
=&~ \frac{2}{\varepsilon}\mathbb E\left[ U'(S_1) \left(\frac{1+\varepsilon - S_1-\varepsilon}{\varepsilon}\right)\right]
+\frac{2}{\varepsilon^2}\widehat\Delta(1+\varepsilon)\mathbb E\left[ U'(S_1) \left(S_1 - 1\right)\right]\\
=&~\frac{2}{\varepsilon^2}\mathbb E\left[ U'(S_1) \left({1- S_1}\right)\right]
+\frac{2}{\varepsilon^2}\widehat\Delta(1+\varepsilon)\mathbb E\left[ U'(S_1) \left(S_1 - 1\right)\right]\\
=&~0,
\end{split}
\end{equation}
where in the last  equality we used \eqref{6281}. Therefore,   \eqref{6293} becomes 
\begin{equation}\label{6297}
\begin{split}
\tfrac{2}{\varepsilon^2}\left(u(1 + \varepsilon) - u(1) - \varepsilon u'(1)\right)
= &~
\mathbb E\left[ U''(\eta(\varepsilon))\frac{\left(\widehat X_1(1 + \varepsilon) - S_1\right)^2}{\varepsilon^2}\right].  \\
\end{split}
\end{equation}
Next, we look at
\begin{equation}\label{6299}
\limsup\limits_{\varepsilon\uparrow 0}\mathbb E\left[ U''(\eta(\varepsilon))\frac{\left(\widehat X_1(1 + \varepsilon) - S_1\right)^2}{\varepsilon^2}\right].
\end{equation}
Using that $u(1)<\infty$, \cite[Lemma 3.6]{KS} together with the symmetry between the primal and dual problems (see the abstract theorems from \cite{Mostovyi2015}) allow to conclude that $\widehat X_1(1+\varepsilon)\to \widehat X_1(1) = S_1$, as $\varepsilon\to 0$, in probability. Consequently, since $\eta(\varepsilon)$ takes values between $\widehat X_1(1)$ and $\widehat X_1(1+\varepsilon)$, we deduce that $\eta(\varepsilon)\to S_1$ in probability. Passing to the limit along a subsequence in \eqref{6299}, and applying Fatou's lemma, we get
\begin{equation}\label{63010}
\limsup\limits_{\varepsilon\uparrow 0}\mathbb E\left[ U''(\eta(\varepsilon))\frac{\left(\widehat X_1(1 + \varepsilon) - S_1\right)^2}{\varepsilon^2}\right] 
\leq 
\sup\limits_{\widetilde \Delta\geq 1}\mathbb E\left[ U''(S_1)\left( \widetilde \Delta S_1 +1 - \widetilde \Delta\right)^2\right],
\end{equation}
where  we used the representation $\frac{\widehat X_1(1 + \varepsilon) - S_1}{\varepsilon} = \widetilde \Delta S_1 +1 - \widetilde \Delta$, for some constant $\widetilde \Delta$, which,  for $\varepsilon\in(-1,0)$,  is bounded from below by $1$,  and the observation that, on  $\widetilde \Delta > 2$,  $U''(\eta(\varepsilon))\left( \widetilde \Delta S_1 +1 - \widetilde \Delta\right)^2$  is  monotone in $\widetilde  \Delta$, for every sufficiently small and negative $\varepsilon$  and for every $\omega$.  

Combining \eqref{6297} and  \eqref{63010}, we deduce that 
\begin{equation}\label{62911}
\limsup\limits_{\varepsilon\uparrow 0}\tfrac{2}{\varepsilon^2}\left(u(1 + \varepsilon) - u(1) - \varepsilon u'(1)\right)
\leq 
\sup\limits_{\widetilde \Delta\geq 1}\mathbb E\left[ U''(S_1)\left(\widetilde \Delta S_1 +1 - \widetilde \Delta\right)^2\right].
\end{equation}
On the other hand, there exists a constant $\varepsilon_0'>0$, such that for every $\bar\Delta \in [-1-2/\varepsilon_0',1]$, we have that 
\begin{equation}\label{8261}
\mathcal X(1+\varepsilon)\ni  X^{\varepsilon,\bar\Delta} := 1+\varepsilon + (1+ \bar\Delta \varepsilon)(S-1),\quad for~every\quad\varepsilon \in(0,\varepsilon_0'].
\end{equation}
In particular, for every $\bar \Delta<1$, we can choose  $\varepsilon_0'$ such that \eqref{8261} holds.
We then obtain
\begin{equation}\label{6302}
\begin{split}
&\liminf\limits_{\varepsilon\downarrow 0}\tfrac{2}{\varepsilon^2}\left(u(1 + \varepsilon) - u(1) - \varepsilon u'(1)\right)\\
 \geq&
  \liminf\limits_{\varepsilon\downarrow 0}\tfrac{2}{\varepsilon^2}\left(\mathbb E\left[U\left(X^{\varepsilon,\bar\Delta}_1\right)\right] - u(1) - \varepsilon u'(1)\right), 
  \end{split}
\end{equation}
where $\mathbb E\left[U\left(X^{\varepsilon,\bar\Delta}_1\right)\right]$ is well-defined see the justification in Remark \ref{remInt}. 
Since $$X^{\varepsilon,\bar\Delta}=S + (\bar\Delta\varepsilon) S  +  \varepsilon (1-\bar\Delta),$$ applying Taylor's expansion once more in \eqref{6302}  gives
\begin{equation}\label{6301}
\begin{split}
\frac{2}{\varepsilon^2}\left(\mathbb E\left[U\left(X^{\varepsilon,\bar\Delta}_1\right)\right] - u(1) - \varepsilon u'(1)\right) 
= &\frac{2}{\varepsilon^2}\left(\varepsilon(\bar\Delta - 1)\mathbb E\left[U'(S_1)(S_1 - 1)\right] \right.\\
&\left.+ \frac{1}{2}\mathbb E\left[U''(\widetilde\eta(\varepsilon))(\bar\Delta\varepsilon S_1 + \varepsilon(1-\bar\Delta))^2\right]\right) \\
=& \frac{1}{\varepsilon^2}\mathbb E\left[U''(\widetilde\eta(\varepsilon))(\bar\Delta\varepsilon S_1 + \varepsilon(1-\bar\Delta))^2\right],
\end{split}
\end{equation}
for some random variable $\widetilde\eta(\varepsilon)$ taking values between $S_1$ and $X^{\varepsilon, \bar\Delta}$, 
and where in the last equality, we have used \eqref{6281}. 

In particular, for $\bar\Delta = \widehat \Delta=-\frac{\mathbb E\left[ U''(S_1)(S_1 -1)\right]}{\mathbb E\left[ U''(S_1)(S_1 -1)^2\right]}$, where  by assumption \eqref{724}, $\widehat \Delta <1$, we can rewrite the latter expression  in \eqref{6301} as
$$\mathbb E\left[U''(\widetilde\eta(\varepsilon))(\widehat\Delta S_1 + (1-\widehat\Delta))^2\right].$$

Note that the function $f(\bar\Delta):= \mathbb E\left[U''(S_1)(\bar\Delta S_1 + (1-\bar\Delta))^2\right]$, $\bar\Delta\in \mathbb R$, reaches its strict global  maximum at $\widehat\Delta$ defined above. Also, from \eqref{8261}, we deduce that $X^{\varepsilon, \bar\Delta}\to S_1$ as $\varepsilon\downarrow 0$. Combining \eqref{6302} and \eqref{6301} and using \eqref{711}, for $\bar \Delta = \widehat \Delta$, we deduce that
\begin{equation}\label{6304}
\begin{split}
&\liminf\limits_{\varepsilon\downarrow 0}\tfrac{2}{\varepsilon^2}\left(u(1 + \varepsilon) - u(1) - \varepsilon u'(1)\right)
\geq 
\mathbb E\left[U''(S_1)(\widehat\Delta S_1 + 1-\widehat\Delta)^2\right].
 \end{split}
 \end{equation}
Therefore, from \eqref{63010} and \eqref{6304}, we conclude that
\begin{equation}\nonumber
\begin{split}
 &\liminf\limits_{\varepsilon\downarrow 0}\tfrac{2}{\varepsilon^2}\left(u(1 + \varepsilon) - u(1) - \varepsilon u'(1)\right)
  \geq
\mathbb E\left[U''(S_1)(\widehat\Delta S_1 + 1-\widehat\Delta)^2\right] \\
= &~ \sup\limits_{\bar \Delta\in\mathbb  R}\mathbb E\left[ U''(S_1)\left( \bar \Delta S_1 +1 - \bar \Delta\right)^2\right]
>  \sup\limits_{\bar \Delta\geq 1}\mathbb E\left[ U''(S_1)\left( \bar \Delta S_1 +1 - \bar \Delta\right)^2\right] \\
\geq&~  \limsup\limits_{\varepsilon\uparrow 0}\tfrac{2}{\varepsilon^2}\left(u(1 + \varepsilon) - u(1) - \varepsilon u'(1)\right),
 \end{split}
 \end{equation}
which shows that 
$u''(1)$ does not exist in the case when $\widehat \Delta <1$, where $\widehat \Delta$ is defined in \eqref{6282}. The case when $\widehat \Delta >1$ can be handled similarly.

We conclude justifying why we may assume that, without loss of generality, $A$ satisfies  $A(1/m)\neq A(1/k)$, for some $m$ and $k$ in $\mathbb N$.
Indeed, for a given utility function $U$, let us consider the family $U_\lambda :=U(\lambda \cdot)$, $\lambda >0$. Then, for a given $\lambda >0$ and every $x>0$, we have
\begin{equation}\label{791}
A_{\lambda}(x) := -\frac{ U_{\lambda}''(x)x}{U_\lambda '(x)}  = -\frac{\lambda^2 U''(\lambda x)x}{\lambda U'(\lambda x)} = A (\lambda x).
\end{equation}
If $A(a)\neq A(b)$ for some $0<a<b$, then, we have,  by \eqref{791}, that $A_{\lambda}(a/\lambda)\neq A_{\lambda}(b/\lambda)$. Therefore, by the choice of $\lambda$, we may assume that 
$$A_{\lambda}(a/\lambda)\neq A_{\lambda}(1/m),$$
for some $m\in\mathbb N$ and where $a/\lambda\in(0,1/m)$. If we add $a/\lambda$ to the range of $S_1$, and assign to this state a  positive but small probability, the arguments above still go through, and imply that $u_\lambda''(1)$ does not exist, where 
$$u_\lambda(x) :=\sup\limits_{X\in\mathcal X(1)}\mathbb E\left[U(\lambda xX_1) \right] = u(\lambda x).$$
Therefore, non-existence of $u_\lambda''(1)$ would imply that $u''(\lambda)$ does not exist either. 
\end{proof}

\section{Second main result: $\mathcal{SD}(2) =\mathcal{SD}(\infty)$}
While condition $\mathcal {SD}(n)$, $n = 2,\dots, \infty$ is the natural condition for differentiability of order $n$, or analyticity, for $n = \infty$, it turns out, when applied to the dual domain, having a maximal element is the same for every $n$. In addition, we have a characterization of such a maximal element. This is related to \cite[Proposition 3.10]{MFT}, but unlike \cite[Proposition 3.10]{MFT}, we do not assume that $\widehat Y$ is a density of a probability measure, as no free lunch with vanishing risk (NFLVR) is not supposed here. Even if NFLVR is assumed, if $\widehat Y$ is maximal in the second-order stochastic dominance, one can conclude that it is a measure. 

We recall that a probability measure $\mathbb Q\sim\mathbb P$ is an {\it equivalent local martingale measure} for $S$  if every $X\in\mathcal X(1)$ is a local martingale under $\mathbb Q$. We will denote the family of equivalent local martingale measures by $\mathcal M^e(S)$. By \cite[Theorem 1.1]{Delbaen-Schachermayer1998}, the celebrated no-free lunch with vanishing risk (NFLVR) condition for $S$  is equivalent to
\begin{equation}\label{NFLVR}
\mathcal M^e_\sigma(S): = \left\{ \mathbb Q\sim\mathbb P:~S~is~a~\sigma-martingale~under~\mathbb Q \right\}\neq \emptyset,
\end{equation}
that is to non-emptiness of the set of {\it equivalent sigma-martingale measures} for $S$. 
Following \cite{Delbaen-Schachermayer1998}, let us also recall that an {\it equivalent separating measure} for $S$ is defined as $\mathbb Q\sim\mathbb P$ such that 
\textcolor{black}{ every $X\in\mathcal X(1)$ is a supermartingale under $\mathbb Q$}. 
$\mathcal M^e_s(S)$ denotes the family of the equivalent separating measures for $S$. 
\begin{rem}\label{remNFLVR}  We recall that, using the Ansel and Stricker Theorem, \cite[Corollary 3.5]{Ansel-Stricker}, one can show that 
$$\mathcal M^e_s(S) \supseteq \mathcal M^e(S) \supseteq \mathcal M^e_\sigma(S)$$
If $\mathcal M^e_s(S)\neq \emptyset$, then $S$ satisfies NFLVR, and thus by \cite[Theorem 1.1]{Delbaen-Schachermayer1998}, we have that $$M^e_\sigma(S)\neq \emptyset.$$ Further, by \cite[Proposition 4.7]{Delbaen-Schachermayer1998}, the density of $\mathcal M^e_\sigma(S)$ in $\mathcal M^e_s(S)$ 
implies that, for every non-negative random variable $g$, we have 
$$\sup\limits_{\mathbb Q\in\mathcal M^e_s(S)} \mathbb {E^Q}[g] = \sup\limits_{\mathbb Q\in\mathcal M^e_\sigma(S)} \mathbb {E^Q}[g],$$
see the proof of \cite[Theorem  5.12]{Delbaen-Schachermayer1998}.
\end{rem}
\begin{prop}
Let us consider a financial model for which  \eqref{NUPBR} holds\footnote{In particular, we do not suppose \eqref{NFLVR}.}. Then the following conditions are equivalent: 
\begin{enumerate}[label=\arabic*.]
\item $\mathcal{SD}(\infty):~\widehat Y_T\succeq_\infty Y_T$, for every $Y\in\mathcal Y(1)$;
\item 
$\widehat Y_T \geq \mathbb E\left[Y_T|\sigma (\widehat Y_T)\right]$, for every $Y\in\mathcal Y(1)$;
\item $\mathcal {SD}(2):~\widehat Y_T\succeq_2 Y_T$, for every $Y\in\mathcal Y(1)$. 
\end{enumerate}
In addition if NFLVR holds, then:
$\mathbb E\left[ \widehat Y_T\right] = 1$ and 
the probability measure $\widehat {\mathbb Q}$ defined through its derivative as $\frac{d\widehat{\mathbb  Q}}{d\mathbb P} := \widehat Y_T$ is a separating measure in the terminology of Delbaen and Schachermayer \cite{Delbaen-Schachermayer1998}, that is $\widehat {\mathbb Q}\in\mathcal M^e_s(S)$.
\end{prop}
\begin{proof}
The implication $3\Rightarrow 1$ is trivial. Likewise, $2 \Rightarrow 3$ follows from Jensen's inequality. 
For the additional statement, under NFLVR there exists $Y_T = \frac{d\mathbb Q}{d\mathbb P}$, $\mathbb Q\in\mathcal M^e_\sigma(S)$, such that $\mathbb E\left[Y_T\right] = 1$, and thus from $2$, we have 
$$\widehat Y_T\geq \mathbb E\left[Y_T|\sigma(\widehat Y_T)\right].$$
Therefore $\mathbb E\left[\widehat Y_T\right]\geq 1$, and the proof is complete.  

It remains to prove $ 1 \Rightarrow   2$. Assume $\mathcal {SD}(\infty)$ and let $\widehat Y$ be the maximal element in $\mathcal Y(1)$ in the sense of the infinite order stochastic dominance. Let us consider  the dual function $V$ such that the following properties hold:
\begin{enumerate}[label=(\roman*)]
\item $-V'$ is $\mathcal{CM}$;
\item the Inada conditions \eqref{Inada}  hold for $-V$;
\item $V$ is bounded.
\end{enumerate}
%

\begin{rem}
One way of constructing such a $V$ is though a measure $\mu$  such that 
\begin{equation}\label{1292}\mu(\{0\}) = 0,\quad  \quad \mu((0,\infty)) = \infty,\quad and \quad \int\limits_0^\infty\int\limits_0^\infty e^{-yt}\mu(dt)<\infty.
\end{equation}
Then $V$ can be defined (through its derivative) as 
\begin{equation}
\label{1291}
V(y) := \int\limits_y^\infty\int\limits_0^\infty e^{-zt}\mu(dt)dz,\quad y>0.
\end{equation}
Then
$$
-V'(y) = \int\limits_0^{\infty} e^{-yt} \mu (dt),\quad y>0,$$
is demonstratively $\mathcal {CM}$. The Inada conditions \eqref{Inada} and finiteness will hold for $-V$ by \eqref{1292}. Thus properties (i)-(iii) hold. Note that, for the finiteness of $V$, we have 
$$\int\limits_0^\infty\left(\int\limits_0^\infty e^{-yt}\mu(dt)\right)dy = \int\limits_0^\infty\left(\int\limits_0^\infty e^{-yt}dy\right)\mu(dt) =\int\limits_0^\infty \frac 1t \mu(dt).$$
This allows for an explicit choice of $\mu$.  
For example, we can pick $\mu$ given by  
\begin{displaymath}
\mu(dt) = \left\{\begin{array}{ll}t^{-\tfrac 12}dt,& t\geq 1\\
0dt,&t\in(0,1)
\end{array}
\right..
\end{displaymath}

Another way to construct  $V$ satisfying (i)-(iii) is the following: for some constants $a \in(-1,0)$ and $b \in (-\infty, 1)$, we can set 
$$V(y) := \int\limits_y^\infty z^a(z+1)^bdz,\quad y>0.$$  
Then, $-V'(y)  = y^a(y+1)^b,$ $y>0$, is $\mathcal {CM}$ as a product of $\mathcal {CM}$ functions, see \cite[Corollary 1.6]{BookBernsteinFunctions}. 
As  $$0\leq V(y)\leq \lim\limits_{y\downarrow 0}V(y)=\int\limits_0^\infty y^a(y+1)^bdz<\infty,$$ and the Inada conditions \eqref{Inada}  hold for $-V'$, we deduce that properties (i)-(iii) hold. 
\end{rem}

Now, with $V$ as above  for the utility function $U(x) = \inf\limits_{y>0}(V(y) + xy)$, $x>0$,  let us consider \eqref{primalProblem}.  Under \eqref{NUPBR}, its dual is \eqref{dualProblem}, where \eqref{finDual} holds by the boundedness of $V$. Therefore,  the abstract theorems in \cite{Mostovyi2015} apply, and we deduce the existence of  $\widehat x = -v'(1)$, such that 
$\widehat X\in\mathcal X(\widehat x)$ and $\widehat X\widehat Y$ is a true martingale and
\begin{equation}\label{1301}
\widehat X_T = -V'(\widehat Y_T).
\end{equation}  

Let us fix a constant $\lambda>0$ and consider a new dual function defined (on the level of the derivative) as 
\begin{displaymath}
V_\lambda(y) := -\int\limits_y^\infty V'(z)\left(1 + e^{-\lambda z}\right)dz,\quad y>0;
\end{displaymath} 
that is 
\begin{displaymath}
-V'_\lambda(y)  = - V'(y)\left(1 + e^{-\lambda y}\right),\quad y>0.
\end{displaymath}  
As $- V'(y)\leq-V'_\lambda (y)\leq - 2V'(y)$, $y>0$, we deduce that the Inada conditions \eqref{Inada} hold for $-V$ and that $V$ is bounded. Further, $-V'_\lambda$ is $\mathcal {CM}$ by \cite[Corollary 1.6]{BookBernsteinFunctions}. Thus $V_\lambda$ satisfies properties (i)-(iii) above, and therefore for $U_\lambda(x) = \inf\limits_{y>0}\left(V_\lambda(y) + xy\right)$, $x>0$, if we consider the utility maximization problem \eqref{primalProblem} and it dual \eqref{dualProblem}, we deduce from \cite[Theorem 3.2]{Mostovyi2015} that 
$x_\lambda := -v'_\lambda(1)$ is well-defined and that there exists $X^\lambda\in\mathcal X(x_\lambda)$ such that $X^\lambda \widehat Y$ is a $\mathbb P$ martingale and 
\begin{displaymath}
X^\lambda_T = -V'_\lambda (\widehat Y_T) = - V'(\widehat Y_T)\left( 1 + e^{-\lambda \widehat Y_T}\right) = \widehat X_T\left( 1 + e^{-\lambda \widehat Y_T}\right),
\end{displaymath}
is the optimizer to \eqref{primalProblem} for $x = x_\lambda$ and the utility function $U_\lambda$. 

Now, one can see that $$x_\lambda - \widehat x>0 \quad and \quad X^\lambda- \widehat X \in\mathcal  X(x_\lambda - \widehat x).$$
Further, $(X^\lambda- \widehat X )\widehat Y$ is a true $\mathbb P$-martingale and 
\begin{displaymath}
\frac{X^\lambda_T- \widehat X_T}{\widehat X_T} = e^{-\lambda \widehat Y_T}.
\end{displaymath}
Let us consider 
\begin{displaymath}
M  := \frac{X^\lambda  - \widehat X }{\widehat X},
\end{displaymath}
and change num\'eraire to $\frac{\widehat X}{\widehat x}$ and the probability measure to  $\mathbb R$ defined as
\begin{displaymath}
\frac{d\mathbb R}{d\mathbb P} = \frac{\widehat X_T\widehat Y_T}{\widehat x}.
\end{displaymath}
One can see that under the num\'eraire $\frac{\widehat X}{\widehat x}$ and measure $\mathbb R$, the sets of the nonnegative wealth processes and supermartingale deflators are given by
\begin{displaymath}
\begin{split}
\widehat {\mathcal X}(x)& := \frac{\mathcal X(x)}{\widehat X}{x} = \left\{\frac{X}{\widehat X}{\widehat x}=\left(\frac{X_t}{\widehat X_t}{\widehat x}\right)_{t\in[0,T]}:~X\in\mathcal X(x)\right\},\quad x>0;\\
\widehat {\mathcal Y}(y)& := \frac{\mathcal Y(y)}{\widehat Y} = \left\{ \frac{Y}{\widehat Y} = \left(\frac{Y_t}{\widehat Y_t}\right)_{t\in[0,T]}:~Y\in\mathcal Y(y)\right\},\quad y>0.
\end{split}
\end{displaymath}
Let us also denote 
\begin{displaymath}
\widehat S:= \left(\frac{\widehat x}{\widehat X}, \frac{\widehat x S}{\widehat X}\right).
\end{displaymath}

One can see that 
$$M\in\widehat {\mathcal X}\left(\frac{x_\lambda - \widehat x}{\widehat x}\right)$$
and that $M$ is a true $\mathbb R$-martingale such that 
$$M_T = e^{-\lambda \widehat Y_T}.$$
Therefore, $M$ is a bounded replication process for $e^{-\lambda \widehat Y_T}$ under  the num\'eraire $\frac{\widehat X}{\widehat x}$ (and measure $\mathbb R$). 

We deduce that for every constant $\lambda >0$, the option 
$$0\leq f_\lambda (\widehat Y_T) := e^{-\lambda \widehat Y_T}\leq 1$$
is {\it replicable} by a bounded stochastic integral  under the num\'eraire $\frac{\widehat X}{\widehat x}$. Let us work under the measure $\mathbb R$ and num\'eraire $\frac{\widehat X}{\widehat x}$. 
One can see that  $\mathbb R\in\mathcal M^e_s(\widehat S)$. In particular (under the num\'eraire $\frac{\widehat X}{\widehat x}$), we obtain that $$\mathcal M^e_s(\widehat S)\neq \emptyset;$$
%
and we have 
\begin{equation}\label{121}
\mathbb E^{\mathbb R}\left[f_\lambda (\widehat Y_T)\right]= \mathbb E^{{\mathbb Q}}\left[f_\lambda (\widehat Y_T)\right],\quad for~every\quad \lambda >0\quad and\quad{\mathbb Q}\in\mathcal M^e_s(\widehat S).
\end{equation}

Next, from every $f_\lambda$ of the form $f_\lambda(y)= e^{-\lambda y}$, $y> 0$, we will extend \eqref{121} to any 
\begin{equation}\label{124}
h:[0,\infty)\to \mathbb R, \quad h~is~continuous\quad and\quad \lim\limits_{y\uparrow\infty}h(y) = 0.
\end{equation}
The latter property of $h$ allows for the Alexandroff extension to $[0,\infty]$. Equivalently, one can consider 
$$g(y) :=\left\{\begin{array}{ll} h(-\log(y)),& y\in(0,1]\\
0, & y = 0 \\
\end{array}\right.,$$ and observe that $g$ is continuous on $[0,1]$. Therefore, we can uniformly approximate $g$ by the polynomials $P_n$ of the form
$$P_n(y) = a_0 + a_1 y+\dots + a_ny^n, \quad y\in[0,1].$$
Since $g(0) = 0$, we further conclude that $a_0 = 0$. Going back to $h$, we deduce that $h$ can be approximated uniformly on $[0,\infty)$ by the functions of the form
$$a_1 e^{-y}+\dots + a_ne^{-ny} = P_n(e^{-y}), \quad y\geq 0.$$
Consequently, for such an $h$ and any $\mathbb Q\in\mathcal M^e_s(\widehat S)$, we get
\begin{equation}
\label{122}
\begin{split}
&\mathbb {E^R}\left[ h(\widehat Y_T)\right] = \mathbb {E^R}\left[ \lim\limits_{n\uparrow\infty}P_n\left(e^{-\widehat Y_T}\right)\right] 
=\lim\limits_{n\uparrow\infty} \mathbb {E^R}\left[ P_n\left(e^{-\widehat Y_T}\right)\right]=\\
=&\lim\limits_{n\uparrow\infty} \mathbb {E^Q}\left[ P_n\left(e^{-\widehat Y_T}\right)\right] = 
\mathbb {E^Q}\left[\lim\limits_{n\uparrow\infty}  P_n\left(e^{-\widehat Y_T}\right)\right] = \mathbb {E^Q}\left[h\left( \widehat Y_T \right)\right],
\end{split} 
\end{equation}
where we have  used  the dominated convergence theorem and \eqref{121}. 
To recapitulate, we have shown that for any function $h$ satisfying \eqref{124}, we have 
\begin{equation}
\label{123}
\mathbb {E^R}\left[ h(\widehat Y_T)\right] =  \mathbb {E^Q}\left[h ( \widehat Y_T  )\right],\quad 
for ~every~\mathbb Q\in\mathcal M^e_s(\widehat S).
\end{equation}

Next, we observe that  \eqref{123} holds for any function $h:$ $[0,\infty) \to \mathbb R$,  which is {\it smooth and has a compact support} in $(0,\infty)$, as every such function satisfies \eqref{124}. Further, using  truncation and regularization by convolution, one can approximate (in the topology of uniform convergence on compact subsets of $(0,\infty)$) any {\it bounded and continuous} function $(0,\infty) \to \mathbb R$  
by a  uniformly bounded sequence of smooth functions with compact support in $(0,\infty)$. Therefore, similarly to the computations in \eqref{122}, we deduce that \eqref{123} holds for every {\it bounded continuous} function $h$: $(0,\infty) \to \mathbb R$.

Let $\mathcal H$ be the set of {\it bounded Borel-measurable} functions $h$: $(0,\infty) \to \mathbb R$, such that \eqref{123} holds. One can see that $\mathcal H$ is a monotone class. By $\mathcal C$  let us denote the set of {\it bounded continuous} functions $(0,\infty) \to \mathbb R$. As $\mathcal C$ is closed under the pointwise multiplication, and $\mathcal C\subseteq\mathcal H$, we deduce from a version of the monotone class theorem, see e.g., \cite[Theorem I.8]{Pr} or \cite[Theorem 21, p. 14]{DelMey78}, that $\mathcal H$ contains all bounded $\sigma(\mathcal C)$-measurable functions.  As $\sigma(\mathcal C)$ is the Borel sigma-field on $(0,\infty)$, we conclude that, for every bounded Borel-measurable function $h$: $(0,\infty) \to \mathbb R$, \eqref{123} holds. Therefore, {\it every  bounded option $h(\widehat Y_T)$ is replicable} under the num\'eraire $\frac{\widehat X}{\widehat x}$.

 Now, let us fix a bounded function $h$: $(0,\infty) \to [0,\infty)$. Then, there exists $x\geq 0$, and $X\in\mathcal X(x)$, such that 
\begin{equation}\label{215}
\widetilde X = \frac{X}{\widehat X}\widehat x\in\widehat{\mathcal X}(x)\quad is~bounded\quad and \quad \widetilde X_T = h(\widehat Y_T).
\end{equation}
We also have 
\begin{equation}\label{1293}
x = \mathbb E^{\mathbb R}\left[h(\widehat Y_T)\right] = \mathbb E\left[ \frac{\widehat X_T}{\widehat x}\widehat Y_T h(\widehat Y_T)\right].
\end{equation}
Now, let us get back to the original num\'eraire and measure: by \eqref{215}, 
we have a process $X\in\mathcal X(x)$, such that $$X_T = \frac{\widehat X_T}{\widehat x}h(\widehat Y_T).$$
Therefore, for any $Y\in\mathcal Y(1)$, we have 
\begin{equation}\label{216}
x\geq \mathbb E\left[X_T Y_T\right] = \mathbb E\left[\frac{\widehat X_T}{\widehat x}h(\widehat Y_T) Y_T\right]
\end{equation}
Comparing  \eqref{1293} and \eqref{216}, we deduce, for every $Y\in\mathcal Y(1)$ and every bounded $h\geq 0$, that 
\begin{displaymath}
\mathbb E\left[\widehat X_T\widehat Y_T h(\widehat Y_T)\right] \geq 
\mathbb E\left[\widehat X_T Y_Th(\widehat Y_T)\right].
\end{displaymath}
Consequently, we have $$
\widehat X_T\widehat Y_T \geq \mathbb E\left[ \widehat X_T\widehat Y_T |\sigma (\widehat Y_T )\right].$$
We recall that by \eqref{1301}, $\widehat X_T = -V'(\widehat Y_T)$, and thus $\widehat X_T$ is $\sigma (\widehat Y_T )$-measurable. We conclude that
$$\widehat Y_T \geq \mathbb E\left[ Y_T|\sigma (\widehat Y_T )\right],$$
for every $Y\in\mathcal Y(1)$.
~ This completes the proof.
\end{proof}
\bibliographystyle{alpha}
\bibliography{literature1}

\end{document}